\documentclass[12pt, leqno]{article}

\usepackage{amsthm, amsfonts, amssymb, color}
 \usepackage{mathrsfs}
\usepackage{amsmath}
 \usepackage{amstext, amsxtra}
  \usepackage{txfonts}
 \usepackage[colorlinks, linkcolor=black, citecolor=blue, pagebackref, hypertexnames=false]{hyperref}

 \allowdisplaybreaks

\newtheorem{theorem}{Theorem}[section]
\newtheorem{proposition}[theorem]{Proposition}
\newtheorem{coro}[theorem]{Corollary}
\newtheorem{lemma}[theorem]{Lemma}

\newtheorem{remark}[theorem]{Remark}

\newcommand\D{\mathcal{D}}
\newcommand\C{\mathbb{C}}
\newcommand\N{\mathbb{N}}
\newcommand\R{\mathbb{R}}

\newcommand\CC{\mathbb{C}}

\newcommand\ZZ{\mathbb{Z}}

\renewcommand\Re{\operatorname{Re}}
\renewcommand\Im{\operatorname{Im}}

\def\SL{\sqrt[m] L}

\newcommand{\supp}{{\rm supp}{\hspace{.05cm}}}

\newcommand\wrt{\,{\rm d}}

\title
{ Spectral multipliers, Bochner-Riesz means
 and uniform Sobolev  inequalities
  for  elliptic  operators
\footnotetext[1]{{\it{\rm 2010} Mathematics Subject Classification:}
Primary: 58J50;
Secondary: 42B15, 42B20, 35P15,  47F05.}
\footnotetext[2]{{\it Key words and phrase:} Spectral multiplier, Bochner-Riesz mean,
uniform Sobolev inequality, spectral measure, restriction type estimate.}}

\smallskip

\author{
Adam Sikora \\
Macquarie University, Sydney\\
sikora@maths.mq.edu.au
\and
Lixin Yan\\
Sun Yat-sen  University, Guangzhou\\
mcsylx@mail.sysu.edu.cn\\
\and
Xiaohua Yao\\
 Central China Normal University, Wuhan\\
yaoxiaohua@mail.ccnu.edu.cn\\
}

 \date{\today}

\begin{document}

\vskip 1cm

\maketitle
\begin{abstract}
This paper comprises  two parts. In the first,
 we study  $L^p$ to $L^q$ bounds for spectral multipliers
and   Bochner-Riesz means with negative index in the general setting
 of abstract self-adjoint operators. In the second we obtain the uniform Sobolev estimates for  constant coefficients higher order elliptic operators $P(D)-z$ and all  $z\in {\mathbb C}\backslash [0, \infty)$, which give  an  extension   of  the second order results of Kenig-Ruiz-Sogge \cite{KRS}. Next we use  perturbation
	 techniques to prove  the uniform Sobolev estimates for  Schr\"odinger operators $P(D)+V$ with small  integrable potentials $V$.
	 Finally  we  deduce spectral multiplier estimates
     for all these  operators, including sharp Bochner-Riesz summability results.
	 
\end{abstract}

 \tableofcontents

\section{Introduction}
\setcounter{equation}{0}

In this paper we investigate  $L^p$ to  $L^q$ estimates  for spectral multiplier operators including
Bochner-Riesz means with negative index in the general setting of abstract self-adjoint operators
as well as elliptic differential operators. We also study closely related issue of the uniform Sobolev estimates. In this
section we review these ideas, present our results, and put them in context.



Suppose  that $X $ is a metric measure space  and that $L$
is a nonnegative self-adjoint operator acting on the
space $L^2(X)$.   Such an operator admits a spectral
resolution  $E_L(\lambda)$. If $F$ is a real-valued  Borel function $F$ on $[0, \infty),$
 we can define the operator $F(L)$ by the formula
\begin{equation}\label{e1.1}
F(L)=\int_0^{\infty}F(\lambda) \wrt E_L(\lambda).
\end{equation}
By the spectral theory the norm $\|F(L)\|_{2\to 2}$ is bounded by $L^\infty$ norm of
the function $F$ (on the spectrum of $L$).
 We call $\wrt E_L(\lambda)$ the spectral measure associated with the operator $L.$
A significant problem often  considered in the spectral multiplier  theory is to describe sufficient
 conditions on   $F$ to ensure
  the boundedness of extension of multiplier  $F(L)$ from the operator defined on   $L^2(X)$ to one
  acting between some $L^p(X)$ spaces or even more general
 functional spaces defined on $X$.
 Since the fundamental works of Mikhlin and H\"ormander on Fourier multipliers \cite{Mi, Ho4},
 one usually looks for  conditions formulated in terms  of differentiability of the function $F$.
  In addition the special instance  of the Bochner-Riesz mean
 described below is also often investigated.

 In the last fifty or so years
 spectral multipliers theory  and the Bochner-Riesz means
have attracted a lot of attention and have been studied extensively by many authors. The existing literature is
too broad to list all  significant contributions to the subject. Therefore  here we mention only some
examples of papers devoted or related to this research area  such as
 \cite{Ale, Blu, C3,  DC2, GS, GH,  Heb, Ho1,  KRS, KU, Ou, T1, St2,
Y-Z}. We wish to point out papers, which investigate sharp spectral
multiplier results, the main focus of our
 study and quote in addition \cite{COSY, ChS, CowS, DOS, GHS, SYY, SS}. We refer the reader for
  references in all works cited  above for more comprehensive list of relevant literature.

One of the  most  significant and more often considered instance of spectral multipliers is
the Bochner-Riesz mean   of
  the operator $L$. To define it,  we put
\begin{equation}\label{e1.2}
 S^{\alpha}_R(\lambda)= \frac{1}{\Gamma(\alpha +1)}\left(1-\frac{\lambda}{R}\right)^{\alpha}_+=\frac{1}{\Gamma(\alpha +1)}
      \left\{
       \begin{array}{cl}
     \left(1-\frac{\lambda}{R}\right)^{\alpha}  &\mbox{for}\;\; \lambda \le R \\ [8pt]
       0  &\mbox{for}\;\; \lambda > R. \\
       \end{array}
      \right.
   \end{equation}
 Then, we call  the operator $S^{\alpha}_R(L)$ defined by \eqref{e1.1}  the Bochner-Riesz mean of order $\alpha.$
  The additional factor $\frac{1}{\Gamma(\alpha +1)}$ is just
reparametrization for positive $\alpha$ which is convenient to use if one consider negative
range of $\alpha$, see \cite{Bak, Bor}.
The case  ${\alpha}=0$ corresponds to the spectral projector $E_{L}([0, R])$, while
for
${\alpha}>0$ one can  think of (\ref{e1.2}) as a smoothed version of the
spectral projector, where   the magnitude of  ${\alpha}$ increases the order of smoothness.
In   perspective the Bochner-Riesz means of the operator $L$ which we develop here
 play a special role because they are not only
the aim of our study but also a crucial tool in this paper.

To be able to  describe and discuss our results we have to introduce  some standard notation.
Throughout this paper we assume that
 $(X,d,\mu)$ is a metric measure  space with a Borel measure  $\mu$.
We denote by
$B(x,\rho)=\{y\in X,\, {d}(x,y)< \rho\}$  the open ball
with centre $x\in X$ and radius $\rho>0$. We   often just use $B$ instead of $B(x, \rho)$.
Given $\lambda>0$, we write $\lambda B$ for the $\lambda$-dilated ball
which is the ball with the same centre as $B$ and radius $\lambda\rho$.
We set $V(x,\rho)=\mu(B(x,\rho))$ the volume of $B(x,\rho)$.

We say that $(X, d, \mu)$ satisfies
 the doubling condition, see  \cite[Chapter 3]{CW},
if there  exists a constant $C>0$ such that
\begin{equation}\tag{${\rm D}$}
V(x,2\rho)\leq C V(x, \rho)\quad \forall\,\rho>0,\,x\in X. \label{e2.1}
\end{equation}
If this is the case, then there exist  constants $n$ and $C$ such that for all $\lambda\geq 1$ and $x\in X$
\begin{equation}\tag{${\rm D_n} $}
V(x, \lambda \rho)\leq C\lambda^n V(x,\rho). \label{e2.2}
\end{equation}
In the sequel we want to consider $n$ as small as possible and we always assume  that
condition \eqref{e2.1} and
\eqref{e2.2} are valid.
In the standard Euclidean space with the Lebesgue measure $n$ coincides with
its dimension.

Next we describe the notion of
Davies-Gaffney estimates, see \cite{BK1, BK3, CouS}.
 Given a  subset $E\subseteq X$, we  denote by  $\chi_E$   the characteristic
function of   $E$ and  set
$$
P_Ef(x)=\chi_E(x) f(x).
$$
Consider again a non-negative self-adjoint operator $L$  and an exponent
$m\geq 2$.
We say that the semigroup $e^{-tL}$ generated by  $L$ satisfies
 { $m$-th order Davies-Gaffney  estimates},
if there exist constants $C, c>0$ such that
\begin{equation}\tag{${\rm DG_{m}}$}
\big\|P_{B(x, t^{1/m})} e^{-tL} P_{B(y, t^{1/m})}\big\|_{2\to {2}}\leq
C   \exp\Big(-c\Big({d(x,y) \over    t^{1/m}}\Big)^{m\over m-1}\Big)\label{DG}
\end{equation}
for all $t>0$  and  $x,y\in X$.

Another condition which we usually impose on the semigroup generated by $L$ can be described in the
 following way.  We   assume that for some $1\le p< 2$,
\begin{equation}\tag{${\rm G_{p,2,m}}$}
\big\|e^{-t^mL}P_{B(x, s)}\big\|_{p\to 2} \leq
CV(x, s)^{{1\over 2}-{1\over p}} \left({s\over  {t}}\right)^{n({1\over p}-{1\over 2})}
\label{GP}
\end{equation}
holds for all $x\in X$ and  $s\geq t>0$.

Recall that  the semigroup  $e^{-tL}$ generated by $L$ is said to  satisfy
$m$-th order (pointwise)  Gaussian estimate ${\rm GE_m}$, see for instance  \cite[ Proposition 2.9]{BK1}, if
semigroup $e^{-tL}$ has integral kernels $p_t(x,y)$ and
there exist constants $C, c>0$
\begin{equation}\tag{${\rm GE_m}$}\label{GE} 
 |p_t(x,y)|\leq {C\over V(x,t^{1/m})} \exp\Big(-c\Big({d^m(x,y) \over    t}\Big)^{1\over m-1}\Big)
\end{equation}
for all $t>0$ and $x,y\in X$.
It is not difficult to note that  both conditions  \eqref{DG}
and \eqref{GP}  for any $1\le p< 2$ follow from Gaussian estimates \eqref{GE}.
 On
the other hand, there are many operators which satisfy
Davies-Gaffney estimates \eqref{DG} for which the standard pointwise
Gaussian estimates \eqref{GE} fail. For example,
Schr\"odinger operators with inverse-square potential see \cite{CouS, ScV}, second order elliptic operators with rough  lower order terms, see \cite{LSV}, or
 higher order elliptic operators with bounded measurable coefficients, see
 \cite{D2}.

%
For  semigroups generated by differential operators the parameter $m \ge 2$ above  usually corresponds to  their order.
The above estimates especially in the case $m=2$  are the main focus of heat kernel theory.
It is a  well-established area of mathematics,  which provides a deep understanding of Gaussian estimates
\eqref{GE} and   a broad class of examples ( operators and  ambient spaces ), for which such estimates hold,   see e.g. Davies \cite{Dav89}, Ouhabaz \cite{Ou} and Grigor'yan \cite{Gr}
and literature therein.

It is known  that conditions
\eqref{DG} and \eqref{GP}
 for some $1\leq p<2$ imply that  the spectral operator $F(L)$ is
bounded on $L^r(X)$ for all $p<r<p'$
for  any bounded Borel
function $F: {\mathbb R_+}\to {\mathbb C}$ such that
 \begin{eqnarray}\label {e1.3}
 \sup_{t>0}\|\eta F(t\cdot)\|_{C^k}<\infty
\end{eqnarray}
 for some $k> n(1/p-1/2)$, see for example  \cite{Ale, Blu, DOS} and Proposition \ref{prop2.2} below.
 Here $\eta \in C_c^\infty(0,\infty)$ is an arbitrary  non-zero auxiliary function.
 In particular, if $p=1$ then  this  corresponds to a
spectral multiplier version of the classical  Mikhlin theorem.

In general spectral multiplier theorems  based on norm in  \eqref{e1.3} do not lead to  critical exponent for
Bochner-Riesz summablity. To obtain such sharp results
 weaker   Sobolev norms
 $W^{\alpha, q}(\R)$ ($1\le q\le \infty$)  are  considered, see \cite{COSY, DOS, SYY}.
 The general perspective of these papers is that condition  ${\rm (ST^{q}_{p, 2, m})}$ below, imply sharp $W^{\alpha, q}(\R)$ version of spectral multipliers.
  See also \cite{GHS, KU}. Condition ${\rm (ST^{q}_{p, 2, m})}$ is motivated by classical
  Stein-Tomas restriction theorem, see \eqref{e1.4} below.
 We point out that considering different values of  $q\in [1,\infty]$ are often essential  for applications. If $q=\infty$, then
conditions ${\rm (ST^{\infty}_{p, 2, m})}$ and \eqref{GP} are equivalent for every $1\le p<2$.
The case $q=2$, corresponding to   classical H\"ormander  theorem, has
another characterization in terms of the spectral measure $dE_L(\lambda)$. Namely, given any $1\le p<2$
condition ${\rm (ST^{2}_{p, 2, m})}$
is equivalent to   the following  estimate
 \begin{eqnarray}
 \label{e1.4}
 \|dE_L(\lambda)\|_{p\to p'}\le C\ \lambda^{\frac{n}{m}(\frac{1}{p}-\frac{1}{p'})-1}, \ \ \lambda>0.\
  \end{eqnarray}
It is a remarkable fact that estimate \eqref{e1.4}  not only play a crucial role in
spectral multiplier theory and Bocher-Riesz analysis
 but can be also regarded as a significant example of  restriction type results in harmonic analysis. Indeed, if $\Delta$
is the standard Laplace operator in $\mathbb R^n$,
 then a $T^*T$ argument yields
  $$
d E_{\Delta}(\lambda) =(2\pi)^{-n} \lambda^{(n-1)/2}R_\lambda^*R_\lambda
 $$
where $R_\lambda$ is the restriction operator  defined by relation
 $
 R_\lambda(f)(\omega) =\hat{f}(\sqrt{\lambda} \omega),
 $
where  $\hat{f}$ is the Fourier transform of $f$ and $\omega\in   {\bf S}^{n-1}$ (the unit sphere).
Thus it follows from the celebrated  {  Stein-Tomas theorem}, see \cite{Tom} that
 the spectral projection measure
 $dE_{-\Delta}(\lambda)$  is bounded as an operator acting  from
 $L^p(\R^n)$ to $L^{p'}(\R^n)$ for any $1\le p\le 2(n+1)/(n+3)$.
 Let us also mention that in \cite{GHS} spectral estimate \eqref{e1.4}
 was obtained in the setting of Laplace type operator acting on asymptotically conic manifolds.

\bigskip

As we said in Abstract this paper comprises two parts.
In  first we will  study  the $L^p \to  L^q$ mapping properties  of spectral multipliers and
Bochner-Riesz means with negative index in the general setting
 of abstract self-adjoint operators.
 In the case of standard Laplace operator and Fourier transform such negative index means were
 studied in  \cite{Bak, Bor, CS, Gu}.
 In our discussion we consider condition
 that Bochner-Riesz mean $S_R^{\alpha}(  \sqrt[m]L)$ satisfies the $(p,q)$-estimate, if there
 exists a constant $C>0$ such that for all $R>0$
$$
\|S_R^{\alpha}(  \sqrt[m] L)P_{B(x,\rho)}  \|_{p\to q} \leq C  V(x,\rho)^{{1\over q}-{1\over p}}
(R\rho)^{{n}({1\over p}- {1\over q}) }
 \leqno{\rm (BR^{ \alpha}_{p, q, m})}
$$
for all $x\in X$ and all $\rho\geq 1/R$, see Section 3 below. In this context, we will show that,
on an abstract level,  Bochner-Riesz means with negative index
 can be used to study spectral multipliers. Roughly speaking,
under the assumption that  $V(x, \rho)\geq C\rho^n$ for all $x\in X$ and $\rho>0$, if $L$
satisfies  Davies-Gaffney  estimates \eqref{DG}
  and   ${\rm (G_{p_0, 2, m})}$ for some $1\leq p_0<2$, and
  ${\rm (BR^{ \alpha}_{p, q, m})}$  for $\alpha\geq -1$ and
$p_0< p< q< p'_0$,
then  for any  $F \in  W^{\beta, 1}(\R)$  such that  $\supp F\subseteq [1/4, 4]$,
 the operator
$F(t\sqrt[m]{L})$ is bounded from  $L^r(X)$ to $L^s(X)$
 for all  $p\leq r  \leq s\leq q$
and $\beta> n(1/p-1/r) +n({1/s}-{1/q}) +\alpha+1$.

 As an application we establish  the $L^p \to  L^q$ mapping properties  of
Bochner-Riesz means of the operator $L$ with negative indexes. To be able to provide more detail description
of our results we  introduce additional notation, which partially coincides with one considered in \cite{BMO, Bor}.
Given some $1\leq p< 2$, we set
\begin{eqnarray*}
A=\left( 1, {n+1+2\alpha\over 2n}\right),&& A'=\left({n-1-2\alpha\over 2n}, 0\right),\nonumber\\
B(p)=\left({n+1+2\alpha\over 2n}+\alpha-{2\alpha\over p}, {n+1+2\alpha\over 2n}\right),&&
B'(p)=\left({n-1-2\alpha\over 2n},  {2\alpha\over p}-\alpha+{n-1-2\alpha\over 2n}\right),\nonumber\\
C(p)=\left({1 \over p}, {n+1+2\alpha\over 2n}\right),&&
C'(p)=\left({n-1-2\alpha\over 2n},   1-{1 \over p}\right),\\
D(p)=\left({1\over 2} +{ \alpha }-{2 \alpha\over  p }, {1 \over 2 }\right),&&
D'(p)=\left({1 \over 2 }, {1\over 2}-{  \alpha }+{2 \alpha\over  p }\right).\nonumber
\end{eqnarray*}
Denote by
$\Delta_{\alpha}(p, n)$  the open pentagon with vertices $A, B(p), B'(p), A', (1,0)$.
In Section 3 we will show that
if we assume  that   $C^{-1}r^n\leq V(x, r)\leq Cr^n$ for all $x\in X$ and $r>0$,
and that  $L$ satisfies   estimates \eqref{DG}, ${\rm (G_{p_0, 2, m})}$ for some $1\leq p_0<2$
 then for any 
 $p_0< p< 2$,
$$
{\rm (BR^{-1}_{p, p', m})}  \Rightarrow   {\rm (BR^{ \alpha}_{r, s, m})}
$$
if each of the following conditions holds:
\begin{itemize}
\item[(1)]   $  \alpha>  n(1/p-1/2)-1/2$,    $ p_0<r\leq s<p_0'$, $ r <q_\alpha$ and $q_\alpha' < s$ where
$q_\alpha=\max\{1,\frac{2n}{n+1+2\alpha}\}$.

\item[(2)]   $ n(1/p-1/2)-1/2 \ge  \alpha> 0$, $p_0<r\leq s<p_0'$,   $(1/r, 1/s)\in \Delta_\alpha(p,n)$
 and $(1/r, 1/s)$ is strictly below the lines joining  the point $(1/2,1/2)$ to $C(p)$  and $C'(p)$.

\item[(3)]   $ -1/2< \alpha\leq 0$, $p_0<r\leq s<p_0'$,   $(1/r, 1/s)\in \Delta_\alpha(p,n)$ and
  $(1/r, 1/s)$ is strictly below  the lines joining $D(p)$ to $C(p)$; $D(p)$ to $D'(p)$ and $D'(p)$ to $C'(p)$.

\item[(4)]      $-1<\alpha\leq  -1/2 $,     $ p_0<r\leq s<p_0'$,  $ \alpha- \frac{2\alpha}{p}< \frac{1}{r}-\frac{1}{s}$,
 $r<q_\alpha'$ and
 $q_\alpha <s$,
 where
$
1/q_\alpha=
 1+\alpha -{(2\alpha+1)/p}.
$
\end{itemize}
In our paper  we do not investigate the endpoint type results.
The perspective developed in \cite{COSY} suggests that such endpoint
estimates can only be obtained  in the  second order case $m=2$.
\bigskip

Next consider $D=-i(\partial_1, \ldots, \partial_n)$ and  operator $P(D)$ where
$P$ is a real elliptic polynomial of oder $m\ge 2$. The  second part of this paper is
 devoted to  restriction type estimates and Bochner-Riesz means of negative order of
differential operators $P(D)+V$, where   $V(x)$ are nonnegative potentials, see Sections \ref{sec4}-\ref{sec6} below.
In the sequel, we write
  $H_0=P(D)$ and $H=P(D)+V$. If $0\le V\in L^1_{\rm loc}(\R^n)$, then it is well
 known that  $H_0$ and $H$ can be defined as nonnegative self-adjoint operators on $L^2(\R^n)$.
Our approach to investigation of  spectral multiplier operators $F(H)$ is to  obtain
 the restriction type estimates \eqref{e1.4} for  $H$. We are able to do this  under the
 standard { non-degenerate condition} of the homogeneous  elliptic polynomial $P(\xi)$ on $\R^n$:
\begin{eqnarray}\label{e1.5}
{\rm det}\ \Big(\ \frac{\partial^2 P(\xi)}{\partial\xi_i\partial\xi_j}\ \Big)_{n\times n}\neq 0,\ \ \  \xi\neq 0.
\end{eqnarray}
The above condition  is equivalent to the fact  that  the compact smooth hypersurface
$\Sigma=\{\ \xi\in \mathbb{R}^n;\ \ |P(\xi)|=1\ \}$
has nonzero Gaussian curvature everywhere.
 In terms of the Fourier transform
we can express the spectral decomposition of  $H_0$ by the following formula
\begin{eqnarray*}
dE_{H_0}(\lambda)f=\big(\delta(P-\lambda)\widehat{f}\ \big )^{\vee}=\frac{1}{(2\pi i)^n}
\int_{\sqrt[m]{\lambda}\Sigma}e^{ix\xi}\widehat{f}(\xi)\ \frac{d\sigma_\lambda(\xi)}{|\nabla P|}.
\end{eqnarray*}
Hence based on the non-degenerate assumption (\ref{e1.5}), for any $1\le p\le 2(n+1)/(n+3)$,
the spectral measure estimates \eqref{e1.4} for $ H_0$  follow from restriction
theorem on the general surface $\Sigma$, see e.g. \cite{G} and Stein \cite[P. 364]{St2}.

For a non-trivial potential $V$  one can not use the Fourier transform to obtain description of spectral resolution of the operator
  $H=H_0+V$. Therefore  we have to develop another perspective to analyse the spectral properties of  $H$, which is  based on
perturbation techniques and some of ideas developed in Section \ref{sec3}.  In our approach we use  Stone's formula:
\begin{eqnarray}\label{e1.7}
dE_{H}(\lambda)f=(2\pi i)^{-1}\Big(R_H(\lambda+i0)-R_H(\lambda-i0)\Big)f, \ \ \lambda>0,
\end{eqnarray}
where $R_H(\lambda\pm i0)$ are defined as  boundaries of the resolvent $(H-z)^{-1}$ of
$H$ with $z\in \mathbb{C}/[0,\infty)$.  To obtain the required bound of $R_H(\lambda\pm i0)$,
 we  establish the following { uniform Sobolev type estimate } for the operator  $P(D)$
\begin{eqnarray}\label{e1.8}
\big\|u\big\|_{L^q(\R^n)}\le C\ |z|^{\frac{n}{m}(\frac{1}{p}-\frac{1}{q})-1}\
\big\|\big(P(D)-z\big)u\big\|_{L^p(\R^n)}, \ \ z\neq 0,
\end{eqnarray}
where $n>m\ge 2$ and the pairs $(p, q)$ satisfy  the following conditions:
\begin{eqnarray}\label{e1.9}\min\Big(\frac{1}{p}-\frac{1}{2},\
\frac{1}{2}-\frac{1}{q}\Big)> {1\over 2n},\ \
\frac{2}{n+1}<\Big(\frac{1}{p}-\frac{1}{q}\Big)\le \frac{m}{n},
\end{eqnarray}
see Corollary \ref{boundary operator} below.
Note that on the  Sobolev embedding line $1/p-1/q=m/n$
estimate \eqref{e1.8} does not contain the term which depends on $|z|$, which means that  it is
uniform for all $z\in \mathbb{C}$ as its name suggests. The proof of
\eqref{e1.8} is based on analysis of oscillatory integral operator
related to restriction theorem, see e.g.  \cite{So}, which  essentially
relies on the non-degenerate curvature condition on the hypersurface $\Sigma$ above.

In the  case $P(D)=-\Delta$,    estimate
\eqref{e1.8} and their more general non-elliptic variants  were obtained  by Kenig, Ruiz and Sogge
and motivated by certain  unique continuation theorems for the operators $P(D)$, see \cite{KRS}.
Here using  \eqref{e1.8} and the following perturbed resolvent identity
\begin{eqnarray}\label{e1.10}R_H(\lambda\pm i0)=R_{H_0}(\lambda\pm i0)\Big(I+VR_{H_0}(\lambda\pm i0)\Big)^{-1},\ \ \lambda>0,
\end{eqnarray}
we will verify  $L^p$-version of  the  { limiting absorption
principle} \eqref{e1.7} for $H$.
Some versions of \eqref{e1.10} and the  { limiting absorption
principle} were  used   by  Agmon in his celebrated scattering work \cite{Ag} on different
weighted subspaces of $L^2(\R^n)$.

In Theorem \ref{th5.6} below, based on  the limiting absorption
principle and uniform Sobolev estimate
we prove that there exists a constant $c_0>0$ 
such that if
\begin{eqnarray}\label{e1.11}
 \|V\|_{n\over m}+\sup_{y\in \mathbb{R}^{n}}
 \int_{\mathbb{R}^{n}}\frac{|V(x)|}{|x-y|^{n-m}}dx\le c_0, \end{eqnarray}
then the spectral measure estimates \eqref{e1.4} hold for $H=P(D)+V$ for all
$ 1\le p <\min \Big(\frac{2(n+1)}{n+3},{n\over m}\Big)$. Note that
when $m=2$ and $n\ge3$ then  the range of $p$ is the same as for  the standard Laplace
operator,  see Remark \ref{remark5.9} below.
Note also that if  $V\in L^{\frac{n}{m}-\varepsilon}\cap L^{\frac{n}{m}+\varepsilon}$
then the expression described in \eqref{e1.11} is finite. This provides a large class of
  rough potentials $V$ to which our result can be
 applied. It is an interesting question whether it is enough to assume that the expression defined by
 \eqref{e1.11} is finite instead of being small enough. It is plausible to expect that further sophistication
 of our approach can lead to result of this type but we are not going to study this issue here.

%


  In our approach we need to assume that the semigroup   $e^{-t H}$ generated by $H=P(D)+V$, satisfies estimates \eqref{DG} or \eqref{GE}.
  Now, since $V(x)$ is nonnegative, it is
  comparably easy to show the { Davies-Gaffney estimates} \eqref{DG}
  see Lemma \ref{le6.1} below. In addition  if $m>n$ or $m=2$, then it
  is well-known that the { Gaussian estimates } \eqref{GE} for $e^{-tH}$ always hold
  for all $0\le V\in L^1_{\rm loc}$, see e.g \cite{B-D}. On the other hand, if $4\le m\le n$,
  then generally, the Gaussian bound of $e^{-tH}$ may  fail to hold. We describe some results of this type in Section \ref{sec6}
   but do not discuss here all relevant  details, instead  we refer the reader to \cite{D2, DDY}.

 The layout of the paper is as follows.  In Section   2     we recall some basic properties of
heat kernels  and state
some known  spectral multiplier results.
 In   Section 3 we will show that,
at an abstract level,  Bochner-Riesz means with negative index
implies spectral multiplier estimates corresponding
to functions supported in dyadic intervals, which  can be used to study
 $L^p \to  L^q$ mapping properties  of   Bochner-Riesz means with negative index.
In Section  4   we  prove the uniform Sobolev estimate \eqref{e1.8}
	 for constant coefficient higher order elliptic operators  on ${\mathbb R}^n$. We then   use the standard  perturbation technique
	  to obtain  estimates of
     spectral projectors  for   elliptic operators $P(D)+V$ with certain potentials $V$ on ${\mathbb R}^n$ in Section 5.
	 From this, we can deduce spectral multiplier estimates
     of these elliptic operators, including Bochner-Riesz summability results
	 in Section 6.

Throughout, the symbols ``$c$" and ``$C$" will denote (possibly
different) constants that are independent of the essential
variables.


\section{Preliminaries }\label{sec2}
\setcounter{equation}{0}

 In this section  we discuss  some basic properties of Gaussian, Davies-Gaffney and Stein-Tomas type estimates.  For $1\le p\le+\infty$, we denote the
norm of a function $f\in L^p(X,{\rm d}\mu)$ by $\|f\|_p$, by $\langle . ,  . \rangle$
the scalar product of $L^2(X, {\rm d}\mu)$, and if $T$ is a bounded linear operator from $
L^p(X, {\rm d}\mu)$ to $L^q(X, {\rm d}\mu)$, $1\le p, \, q\le+\infty$, we write $\|T\|_{p\to q} $ for
the  operator norm of $T$.
For a given  function $F: {\mathbb R}\to {\mathbb C}$ and $R>0$, we define the
function
$\delta_RF:  {\mathbb R}\to {\mathbb C}$ by putting
 $\delta_RF(x)= F(Rx).$  Given $p\in [1, \infty]$, the conjugate exponent $p'$
is defined by $1/p +1/p' =1.$

For a function $W:M\rightarrow \mathbb{R}$,  let $M_{W}$ the
operator of multiplication by $W$, that is
$$(M_{W}f)(x)=W(x)f(x).$$ In the sequel, we shall identify the function $W$ and the operator $M_W$. That is, if $T$ is a linear operator, we shall denote by $W_1TW_2$ the operator  $M_{W_1}TM_{W_2}$.
We also set $V^{\alpha}_t(x)=V(x,t)^{\alpha}$.


\medskip

\noindent
\noindent
\subsection{Gaussian   estimates and Davies-Gaffney estimates}

\begin{proposition}\label {prop2.1}
Let $m\geq 2$ and $1\leq p< 2$.
 Let  $L$  be  a non-negative self-adjoint operator on $L^2(X)$ satisfying
Davies-Gaffney estimates \eqref{DG} and condition \eqref{GP}.
     Then for all $p< r\leq q< p'$ and for all $\alpha, \beta\geq 0$
such that $\alpha+\beta=1/r-1/q$
\begin{eqnarray}\label{e2.4}\hspace{1cm}
 \| V_t^{\alpha} e^{-t^mL}  V_t^{\beta}
\|_{r\to q} \le C,
\end{eqnarray}
and
\begin{eqnarray}\label{e2.5}
 \big\|  V_t^{\alpha} (I+t^m  L)^{-N/m} V_t^{\beta} \big\|_{r \to q}
 \leq  C
\end{eqnarray}
for every  $N>n(1/r-1/q)$.
\end{proposition}

\begin{proof} From condition \eqref{GP}
\begin{eqnarray}\label{e2.6}
  \| P_{ B(x, t)} e^{-t^mL} P_{B(y,t)} \|_{p\to p'} \le C V(x,t)^{{1\over 2}-{1\over p}}V(y,t)^{{1\over p'}-{1\over 2}}.
\end{eqnarray}
Let $\tilde{p}\in (p, 2)$. Note that it follows from the doubling condition \eqref{e2.2} that
\begin{equation*}
V(y,\rho)\leq C\left( 1+{\frac{d(x,y)}{ \rho}}\right)^{n} V(x,\rho) \quad
\forall\,\rho >0,\,x,y\in X
\end{equation*}
From the above estimate, \eqref{e2.6} and  Davies-Gaffney estimates \eqref{DG},
the Riesz-Thorin interpolation theorem give    the following $L^{\tilde{p}}-L^{\tilde{p}'}$ off-diagonal estimate such that
there exist constants $C, c'>0$ such that
for all $t>0$,  and all $x,y\in X,$
\begin{eqnarray}\label{e2.7}\hspace{1cm}
 \| P_{ B(x, t)} e^{-t^mL} P_{B(y,t)}
\|_{u\to u'} \le C V(x,t)^{ \frac{1}{\tilde{p}'} - \frac{1}{\tilde{p}}} \exp\Big(-c'\Big({d(x,y) \over    t }\Big)^{m\over m-1}\Big).
\end{eqnarray}
By (ii) of  \cite[Proposition 2.1]{BK3}, we obtain that for $\tilde{p}\leq r\leq q\leq \tilde{p}'$ and for all $\alpha, \beta\geq 0$
such that $\alpha+\beta=1/r-1/q$
\begin{eqnarray}\label{e2.8}\hspace{1cm}
 \| V_t^{\alpha} e^{-t^mL}  V_t^{\beta}
\|_{r\to q} \le C,
\end{eqnarray}
which proves  \eqref{e2.4}.

Next,   for  $t > 0$,
\begin{eqnarray*}
  (I+t^m  L)^{-N/m}=
 C_N \int_0^{\infty}    e^{-s} s^{{N/m}-1}
 e^{-st^mL}  ds
\end{eqnarray*}
for some $C_N.$ It then follows that
 for every $p<r\leq q<p'$,
\begin{eqnarray}\label{e2.9}
\left\| V_t^{\alpha} (I+t^m  L)^{-N/m} V_t^{\beta}\right\|_{r\to q}
 &\leq&C_N\int_0^{\infty}    e^{-s} s^{{N/m}-1}
 \left\| V_t^{\alpha}  e^{-st^mL}V_t^{\beta} \right\|_{r\to q}{ds}.
\end{eqnarray}
Observe that for every $z\in X$,    if $s<1$, then
$$
V(z, t)\leq Cs^{-n/m}V(z, s^{1/m}t)
$$
 and if $s>1$, then $V(z, t)\leq C V(z, s^{1/m}t)$.
Estimate
(\ref{e2.9})  yields \eqref{e2.5} for   $N>n(1/r-1/q)$.  This ends
  the proof.
 \end{proof}

\subsection{Stein-Tomas restriction type condition}
Let us recall the restriction type estimates
$({\rm ST}^{q}_{p, 2, m})$, which were originally introduced
in~\cite{DOS} for $p = 1$, and then in~\cite{COSY} for
general~$1<p<2$.
Consider a non-negative self-adjoint operator $L$  and exponents
$p $ and $q$ such that $1\leq p< 2$ and $1\leq
q\leq\infty$. Following \cite{SYY}, we say that  $L$ satisfies the
    {  Stein-Tomas restriction type condition}   if
  for any $R>0$ and all Borel functions $F$ such that $\supp F \subset [0, R]$,
$$
\big\|F(\SL)P_{B(x, \rho)} \big\|_{p\to 2} \leq CV(x,
\rho)^{{1\over 2}-{1\over p}} \big( R\rho \big)^{n({1\over p}-{1\over
2})}\big\|\delta_RF\big\|_{q}
\leqno{\rm (ST^{q}_{p, 2, m})}
$$
  for all
  $x\in X$ and all $\rho\geq 1/R$.

As we mentioned in Introduction this condition  is  motivated by analysis of the standard
Laplace operator $\Delta=-\sum_{i=1}^n\partial^2_{x_i}$ on ${\mathbb R^n}$.
It is not difficult to observe, see  \cite[Proposition 2.4]{COSY}, that for $q=2$
the condition ${\rm (ST^{2}_{p, 2, 2})}$ is equivalent to
the $(p,2)$ Stein-Tomas restriction estimate
 $$
 \|dE_{\sqrt{\Delta}}(\lambda)\|_{p\to p'} \leq C\lambda^{n(1/p- 1/p')-1}
 $$
 for all $ 1 \le p\le \frac{2(n+1)}{ n+3}$.

 Note that   if  condition ${\rm (ST^{q }_{p, 2, m})} $ holds for  some $q\in [1, \infty)$,
 then ${\rm (ST^{\tilde{q}}_{p, 2, m})} $ is automatically valid  for all $\tilde{q}\geq q$ including the case
 $\tilde{q}=\infty$.
It is known that if ${q}=\infty$, then the condition  ${\rm (ST^{{\infty}}_{p, 2, m})} $
follows from the standard elliptic estimates, that is, to be more precise the conditions
${\rm (ST^{{\infty}}_{p, 2, m})}$ and \eqref{GP} are equivalent,  
 see for instance \cite[Proposition 2.2]{SYY}.

 \smallskip

We start with stating very general spectral multiplier result.  Point $(i)$ of the following
proposition can be easily applied  in a wide range of situations but usually does not give
the sharp result and the differentiability assumption can often be relaxed. However,
this general statement helps to avoid nonessential technicalities while discussing sharp
spectral multiplier results.  Point $(ii)$ usually leads to optimal results but verifying condition
${\rm (ST^q_{p, 2, m})}$ is quite difficult.
For the proof, we refer the reader to   \cite{Blu} (for point $(i)$)  and \cite[Theorem 5.1]{SYY}
(for both parts).
 Recall that $n$ is the  doubling dimension from   condition \eqref{e2.2} and
  $\eta\in C_c^{\infty}(0, \infty)$ is  a non-zero  auxiliary function and $C_k$ is a space of
  $k$ times continuously differentiable functions on the real line.

 \begin{proposition}\label{prop2.2}
Let  $L$  be  a non-negative self-adjoint operator   on $L^2(X)$ satisfying
Davies-Gaffney estimates \eqref{DG}. Then

\begin{itemize}
\item[(i)] Assume that  the  condition  ${\rm (G_{p, 2, m})}$ holds for some $p$ satisfying
  $1\leq p<2$.  Then for any  bounded Borel
function $F$ such that
 $$
 \sup_{t>0}\|\eta\delta_tF\|_{C^k}<\infty
 $$
 for some integer $k>n(1/p-1/2) $,
  the operator  $F(L)$ is bounded on $L^{r}(X)$ for all $p<r<p'$.

 \smallskip

\item[(ii)]  Assume that   the condition  ${\rm (ST^q_{p, 2, m})}$ holds for some $p, q$ satisfying
  $1\leq p<2$ and $1\leq  q\leq \infty$. Then for any  bounded Borel
function $F$ such that
$\sup_{t>0}\|\eta\, \delta_tF\|_{W^{\alpha}_q}<\infty $ for some
$\alpha>\max\{n(1/p-1/2),1/q\}$,  the operator
$F(L)$ is bounded on $L^r(X)$ for all $p<r<p'$.
 \end{itemize}
\end{proposition}

A significant example of spectral multipliers are  Bochner-Riesz means.
Let us recall that Bochner-Riesz operators of  index $\alpha$ for a non-negative self-adjoint operator $L$
are defined by the formula
\begin{equation*}
 S_{R}^{\alpha}( L) ={1\over \Gamma(\alpha+1)}
 \left(I-{L\over R}\right)_+^{\alpha},\ \ \ \ R>0.
\end{equation*}
\noindent

Bochner-Riesz analysis studies the range of ${\alpha}$
for which the operators  $ S_{R}^{\alpha}( L)$ are uniformly bounded on $L^p$.
Applying spectral multiplier theorems to study boundedness of Bochner-Riesz means is often
an efficient test to check if the considered result is sharp or not.

\begin{coro}\label{coro2.3}
Suppose that  the
 operator $L$ satisfies
Davies-Gaffney estimate  \eqref{DG}  and condition  ${\rm (ST^q_{p, 2, m})}$
with some $1\leq p<2$ and $1\leq q\leq \infty$.
Then for all  $p<r<p'$ and ${\alpha}> n(1/p-1/2)-1/q$,
\begin{eqnarray}\label {e2.11}
\sup_{R>0}\left\|S_{R}^{\alpha}( L)\right\|_{r\to r}\leq C.
\end{eqnarray}
\end{coro}

 \begin{proof}  For the proof, we refer the reader  to  \cite[Corollary 4.4,]{SYY}.
 \end{proof}

\bigskip

\section{Spectral multipliers and Bochner-Riesz means}\label{sec3}
\setcounter{equation}{0}

Assume that $(X,d, \mu)$ satisfies the doubling condition, that is
(\ref{e2.1}).
Suppose $L$
is a nonnegative self-adjoint operator acting on the
space $L^2(X)$. Such an operator admits a spectral
resolution  $E_L(\lambda)$. If $F$ is a real-valued  Borel function $F$ on $[0, \infty)$,
then one can define the operator $F(L)$ by the formula
\begin{equation}\label{e3.1}
F(L)=\int_0^{\infty}F(\lambda) \wrt E_L(\lambda).
\end{equation}
By spectral theory if the function $F$ is bounded, then the operator $F(L)$ is bounded as an
operator acting on $L^2(X)$. Many authors study   necessary conditions on function $F$ to ensure that
$F(L)$ is bounded as operator action on $L^p$ spaces for some range of $p$.
However we are interested here in estimates   of $L^p \to L^q$
norm of $F(L)$ for some $1\le p <q \le \infty$ especial in situation when $F$ is potentially unbounded.

Observe  that for   ${\rm Im} \lambda\not=0$, the resolvent family $(L-\lambda)^{-1}$ is a holomorphic family
of bounded operators on $L^2(X).$ Throughout this article, we assume that:

\smallskip
{\it    The resolvent family
of the operator $L$
extends continuously to the real axis as a bounded operator in a weaker sense, e.g., between weighted
$L^2$-spaces. }

\smallskip

It is then differentiable in $\lambda$ up to the real axis.
This  property  is satisfied  by many operators,
e.g., constant coefficients
higher order elliptic operators $P(D)$ described in Section 4 below.  Under this assumption, we find
 that $E_L(\lambda)$ is differentiable in $\lambda$ and Stone's formula for operator $L$ is valid
\begin{eqnarray*}
{d\over d\lambda} E_{L}(\lambda)={1\over 2\pi i}\left((L-(\lambda+i0))^{-1}-(L-(\lambda-i0))^{-1}\right).
\end{eqnarray*}
In this case we write (abusing notation somewhat)
$\wrt E_L(\lambda)$ for the derivative of $E_L(\lambda)$ with respect to $\lambda.$
Stone's formula gives a mechanism for analysing the spectral measure, namely we need to analyse the limit
of the resolvent $(L-\lambda)^{-1}$ on the real axis, see Sections 4-6 below.

\subsection{Bochner-Riesz means with negative index}\
In the same way in which we defined the Bochner-Riesz means of the operator
$L$ one can also define Bochner-Riesz means of its root of order $m$, that is  $\sqrt[m] L$.
Similarly as before, for every $R>0$ the
Bochner-Riesz means of index  $\alpha$  for   the operator$\sqrt[m] L$
are defined by the formula
\begin{equation}\label{e3.3}
S_{R}^{\alpha} (\SL)  =\frac{1}{\Gamma(\alpha +1)}
\left(I-{ \sqrt[m] L\over R}\right)_+^{\alpha}, \ \ \ \ \ \  \alpha>-1.
\end{equation}
When $\alpha=-1,$ we set $S_{R}^{-1} (\SL)=R^{-1}dE_{\SL}(R).$
Given   some $   \alpha\geq -1 $  and  $1\leq  p< q\leq \infty$,
we say that the Bochner-Riesz mean  $S_R^{\alpha}(  \sqrt[m]L)$ satisfies the $(p,q)$-estimate, if there
 exists a constant $C>0$ such that for all $R>0,$
$$
\|S_R^{\alpha}(  \sqrt[m] L)P_{B(x,\rho)}  \|_{p\to q} \leq C  V(x,\rho)^{{1\over q}-{1\over p}}
(R\rho)^{{n}({1\over p}- {1\over q}) }
 \leqno{\rm (BR^{ \alpha}_{p, q, m})}
$$
for all $x\in X$ and all $\rho\geq 1/R.$

In our first statement of this section  we note that considering the Bochner-Riesz means of the
operators  $L$ and $\sqrt[m]L$ are essentially equivalent under  some
assumptions of the operator $L$.

\medskip
\begin{lemma}\label{le3.1}
Suppose that $(X, d, \mu)$ satisfies the doubling condition  \eqref{e2.1}
and that the semigroup corresponding to
  a non-negative self-adjoint operator $L$  satisfies  estimates \eqref{DG}
  and ${\rm (G_{p_0, 2, m})}$ for some $1\leq p_0<2$.   Then for   every $   \alpha\geq -1$ and
$ 1\leq p_0<p< q< p'_0$,
  ${\rm (BR^{ \alpha}_{p, q, m})} $   is equivalent to
\begin{eqnarray}\label{e3.4}
\left\|S_{R^m}^{\alpha} (  L) P_{B(x,\rho)}  \right\|_{p\to q} \leq C  V(x,\rho)^{{1\over q}-{1\over p}}
(R\rho)^{n({1\over p}- {1\over q}) }
\end{eqnarray}
for all $x\in X$ and all $\rho\geq 1/R.$
  \end{lemma}

 \begin{proof} Let $\varphi$ be a non-zero $C_0^{\infty} $ function on ${\mathbb R}$
 such that $\varphi(s)=1$ if $s\in [-1/2, 3/2]$ and $\varphi(s)=0$ if $|s|\geq 2$.  For $s>0$, we write
\begin{eqnarray}\label{e3.5}
(1-s)^{\alpha}_+=(1-s^{1/m})_+^\alpha
\big(1+\sum_{k=1}^{m-1}s^{k/m}\big)^{\alpha}\varphi(s)
\end{eqnarray}
 and
\begin{eqnarray}\label{e3.6}
(1-s^{1/m})_+^\alpha=(1-s)^{\alpha}_+
\big(1+\sum_{k=1}^{m-1}s^{k/m}\big)^{-\alpha}\varphi(s).
\end{eqnarray}
We apply Proposition  \ref{prop2.2} to obtain that for $p_0<r<p_0'$, there exists a constant $C>0$ independent of $R$
such that
\begin{eqnarray}\label{e3.7}
\left\|\left(1+\sum_{k=1}^{m-1}\left({\sqrt[m]L\over R}\right)^k\right)^{\pm \alpha}
 \varphi \left({ L\over R^m}\right) \right\|_{r\to r}
\leq C.
\end{eqnarray}
This, together with \eqref{e3.5} and \eqref{e3.6}, proves   Lemma~\ref{le3.1}.
 \end{proof}

It is easy to note that for the standard Bochner-Riesz means corresponding to the Fourier transform
 and the standard Laplace operator,  condition    ${\rm (BR^{ \alpha}_{p, q, m})} $
implies the same estimates  for all exponents $r,s$ such that  $ 1<r\leq p< q\leq s< \infty$.
In our next statement we show that this is a quite general situation limited only by a range of
$L^p$ spaces on which the semigroup generated by $L$ acts and enjoys generalised Gaussian estimates.

\begin{lemma}\label{le3.2}
Suppose that  there exists a constant $C>0$ such that  $C^{-1}\rho^n\leq V(x, \rho)\leq C\rho^n$ for all $x\in X$ and $\rho>0$.
Next assume that $L$ is
  a non-negative self-adjoint operator  acting on $L^2(X)$ satisfying estimates \eqref{DG}
  and ${\rm (G_{p_0, 2, m})}$ for some $1\leq p_0<2$ and that $   \alpha\geq -1 $.  Then
   ${\rm (BR^{ \alpha}_{p, q, m})} $ with $ p_0<  p< q< p'_0$
  implies $  {\rm (BR^{ \alpha}_{r, s, m})}$  for  all $ p_0<r\leq p< q\leq s< p'_0$.

In particular, if the operator $L$ satisfies the Gaussian estimate \eqref{GE}, then
  ${\rm (BR^{ \alpha}_{p, q, m})} $ implies $  {\rm (BR^{ \alpha}_{r, s, m})}$  for all
$ 1\leq r\leq  p< q\leq s\leq\infty$.
  \end{lemma}

\begin{proof} We first show that ${\rm (BR^{ \alpha}_{p, q, m})} $ implies $ {\rm (BR^{ \alpha}_{p, s, m})}$  for
  $1\leq p_0\leq p< q< s< p'_0$.
   We choose a function $\varphi\in C_0^{\infty}(-2, 2)$
 such that $\varphi(s)=1$
if $s<1$; $0$ if $s>2$. Let $N>n(1/q-1/s).$
Note that  $V(x, \rho)\geq C^{-1}\rho^n$ for all $x\in X$ and $\rho>0$.
For $p_0<q<s<p'_0,$ it follows by Proposition  \ref{prop2.1}
that
\begin{eqnarray}\label{e3.8}
\left\| \left(1+ {L\over \lambda^m}\right)^{-N/m} \right\|_{q\to s}
&\leq& C \lambda^{{n}({1\over q}-{1\over s})}\left\|V(x, \lambda^{-1})^{{1\over q}-{1\over s}}
\left(1+{  L\over \lambda^m}\right)^{-N/m} \right\|_{q\to s}\nonumber\\
&\leq&  C \lambda^{{n}({1\over q}-{1\over s})}.
\end{eqnarray}
Hence by Proposition  \ref{prop2.2}
\begin{eqnarray}\label{e3.9}
\left\|\varphi\left( {\sqrt[m]L\over\lambda}\right)\right\|_{q\to s}&=&
\left\|\varphi\left( {\sqrt[m]L\over\lambda}\right)\left(1+{  L\over \lambda^m}\right)^{N/m}\right\|_{q\to q}
\left\| \left(1+ {L\over \lambda^m}\right)^{-N/m} \right\|_{q\to s}\nonumber\\
&\leq&   C \lambda^{{n}({1\over q}-{1\over s})}.
\end{eqnarray}
Note that  $ (1-s/\lambda)_+^{\alpha}
 = \varphi(s/\lambda)(1-s/\lambda)_+^{\alpha}$ and $V(x, \rho)\leq C\rho^n$ for all $x\in X$ and $\rho>0$.
It follows that
%
\begin{eqnarray}\label{e3.10}
 \|S_\lambda^{\alpha}( \sqrt[m]L)P_{B(x,\rho)}\|_{p\to s}&=&
 \left\| \varphi\left( {\sqrt[m]L\over\lambda}\right)
S_\lambda^{\alpha}( \sqrt[m]L) P_{B(x,\rho)}\right\|_{p\to s}\nonumber\\
&\leq&\left\|\varphi\left( {\sqrt[m]L\over\lambda}\right)\right\|_{q\to s}
 \|S_\lambda^{\alpha}( \sqrt[m]L)P_{B(x,\rho)}\|_{p\to q}
\nonumber\\
&\leq& C\lambda^{{n}({1\over q}-
 {1\over s})} V(x,\rho)^{{1\over q}-{1\over p}}(\lambda \rho)^{n({1\over p}-
 {1\over q}) }
 \nonumber\\
&\leq& C V(x,\rho)^{{1\over s}-{1\over p}} (\lambda \rho)^{{n }({1\over p}- {1\over s}) }
\end{eqnarray}
since   $V(x, \rho)\leq C\rho^n$ for all $x\in X$ and $\rho>0.$  Hence,
${\rm (BR^{ \alpha}_{p, q, m})} \Rightarrow  {\rm (BR^{ \alpha}_{p, s, m})}$  for
  $1\leq p_0\leq p< q< s< p'_0.$
A similar argument as above shows that ${\rm (BR^{ \alpha}_{p, q, m})} $ implies $ {\rm (BR^{ \alpha}_{r, q, m})}$  for
  $1\leq p_0<r\leq  p< q<p'_0.$

 As we notice before,   condition ${\rm (G_{1, 2, m})}$ follows from   \eqref{GE},
 so the second part of the Lemma ~\ref{le3.2}
 follows from the first part.
This ends the proof of  Lemma~\ref{le3.2}.
\end{proof}

Our next result is a version of Lemma~\ref{le3.2} corresponding to the case $   \alpha=-1 $.
In this situation the proof simplifies and we can omit the Gaussian bounds assumptions from the statement.

\begin{lemma}\label{lemma 3.6} Let a nonnegative self-adjoint operator $H$
satisfy the $(p_0,p'_0)$-restriction estimate for some $1< p_0<2$ such that
 \begin{eqnarray}\label{eq3-05}
 \|dE_{H}(\lambda)\|_{p_0\to p'_0} \leq C\lambda^{\frac{n}{m}(1/p_0- 1/p'_0 )-1}.
\end{eqnarray}
 In addition, if there exists some $k>0$ such that
 \begin{eqnarray}\label{eq3-05a}
 \|(1+tH)^{-k}\|_{p\to p_0}\leq C_k t^{-{n\over m}({1\over p}-{1\over p_0})},  \ t>0,
  \end{eqnarray}
  for  some $1\leq p\le p_0<2$.
 Then  the estimate
 \begin{eqnarray}\label{eq3-05b}
 \|dE_{H}(\lambda)\|_{q\to q'} \leq C\lambda^{{n\over m}({1\over q}- {1\over q'})-1}.
 \end{eqnarray}
holds for all $p\le q\le p_0$.
\end{lemma}

\begin{proof} By interpolation, it suffices to prove the endpoint case $q=p$.
 In fact, we observe that $(1+H/\lambda)^{-2k}dE_{H}(\lambda)=2^{-2k}dE_{H}(\lambda).$
Then by duality it follows that
\begin{eqnarray*}
\|2^{-2k}dE_{H}(\lambda) \|_{p\to p'}&=&
\|(1+H/\lambda)^{-k}dE_{H}(\lambda)(1+H/\lambda)^{-k} \|_{p\to p'}\nonumber\\
&\leq& \|(1+H/\lambda)^{-k}  \|_{p\to p_0}
\|dE_{H}(\lambda) \|_{p_0\to p'_0}
\|(1+H/\lambda)^{-k} \|_{p'_0\to p'}\nonumber\\
&\leq& C\lambda^{\frac{n}{m}({1\over p}-{1\over p_0})} \lambda^{\frac{n}{m}({1\over p_0}-
 {1\over p'_0})-1}\lambda^{\frac{n}{m}({1\over p'_0}-{1\over p'})}
 \nonumber\\
&\leq& C\lambda^{\frac{n}{m}({1\over p}- {1\over p'})-1}.
\end{eqnarray*}
This ends the proof.
\end{proof}

\begin{remark} The resolvent power $(1+tH)^{-k}$ in  condition \eqref{eq3-05a}  can be
replaced by the semigroup $e^{-tH}$, which actually are equivalent  by some standard arguments.
\end{remark}

Next we describe a useful notation of one-dimensional homogeneous
 distributions $\chi_-^a$ and $\chi_+^a$ coming from \cite{Ho1} and defined by
\begin{eqnarray}\label{e3.11}
\chi_\pm^{\alpha}=\frac{x_\pm^a}{\Gamma({\alpha}+1)},\ \ \ \   {\rm Re}\, {\alpha}>-1,
\end{eqnarray}
where $\Gamma$ is the Gamma function and
$$
x_+^{\alpha}=x^{\alpha} \quad \mbox{if} \quad x \ge 0 \quad \quad \mbox{and}
\quad  x_+^{\alpha}=0 \quad  \mbox{if} \quad x < 0;
$$
$$
x_-^{\alpha}=|x|^{\alpha} \quad \mbox{if} \quad x \le 0 \quad \quad \mbox{and}
\quad  x_-^{\alpha}=0 \quad  \mbox{if} \quad x > 0.
$$
It easy to note that $x_\pm^{\alpha}$ are well defined  distributions for $\Re {\alpha} > -1$. From a
straightforward observation $\frac{d}{dx} x_\pm^{\alpha} = \pm{\alpha} x_\pm^{{\alpha}-1},$ it follows that
\begin{equation*}
 \frac{d}{dx} \chi_\pm^{\alpha}=\pm\chi_\pm^{{\alpha}-1}
\end{equation*}
 for all $\Re {\alpha} > 0$.
One can  use the above relation to extend the family of functions $ \chi_-^{\alpha}$ to
 a family of distributions on $\R$ defined for all ${\alpha}\in \C$,
 see \cite[Ch III, Section 3.2]{Ho1} for details.  Since $1-\chi_-^0(x) $ is the Heaviside function, it follows that
\begin{equation*}
\chi_\pm^{-k} =(\pm 1)^k\delta_0^{(k-1)}, \quad  k=1,2,\ldots,
\label{Heaviside}
\end{equation*}
where $\delta_0$ is  the $\delta$-Dirac measure.

A straightforward computation shows that for all $w,z\in \C$
\begin{equation}\label{e3.12}
\chi_-^w * \chi_-^z=\chi_-^{w+z+1},
\end{equation}
where $\chi_-^w* \chi_-^z$ is the convolution of the distributions $\chi_-^w$ and $\chi_-^z$,
see \cite[(3.4.10)]{Ho1}.
If $\supp F \subset [0,\infty)$,  we then define the Weyl fractional
derivative of $F$ of order $\nu$ by the formula
\begin{equation}\label{e3.13}
F^{(\nu)}=F*{\chi}^{-\nu-1}_-, \ \ \ \nu\in{\mathbb C}
\end{equation}
and we note that for every $\nu\in{\mathbb C},$
$$
F^{(\nu)}*{\chi}^{\nu-1}_-=F*{\chi}^{-\nu-1}_-*{\chi}^{\nu-1}_-=F,
$$
see \cite[Page 308]{GP} or  \cite[(6.5)]{DOS}.
It follows from the above equality and Fubini's Theorem that  for every $\nu\geq 0,$
\begin{eqnarray}\label{e3.14}
F(L)&=&\frac{1}{\Gamma(\nu)}\int_0^\infty F^{(\nu)}(s)
(s-{L})_+^{\nu-1}ds
\end{eqnarray}
for all   $F$  supported in $[0, \infty)$. Relation \eqref{e3.14} plays an important role in the proof of
Proposition \ref{prop3.3} below.

At this point it is convenient to introduce slightly modified version of the standard Sobolev spaces.
Namely, if $\supp F \subset [0,\infty)$, then for any $1 \le p \le \infty$ and $\nu \in \R$
we define the Weyl-Sobolev norm of $F$ by the formula
$$
\big\| F \big\|_{WS^{\nu, p}}=\|F\|_p+ \|F^{(\nu)}\|_p.
$$

\begin{remark}
Note that for $1<p< \infty $ the Weyl-Sobolev norm is equivalent to the standard Sobolev
norm, that is
$$
c\| F \|_{W^{\nu, p}} \le \| F \|_{WS^{\nu, p}}\le C \| F \|_{W^{\nu, p}}
$$
whereas for $p=1$
$$
\| F \|_{WS^{\nu, 1}} \le   C_\varepsilon \| F \|_{W^{\nu+\varepsilon, 1}}
$$
for any $\varepsilon >0$.
\end{remark}

\begin{proof}
Note that $
(1-d^2/dx^2)^{-(\alpha+1)/2}I_\alpha$ is an example of classical (one dimensional)  H\"ormander type
Fourier multiplier and is bounded on all $L^p({\mathbb R})$ spaces for $1<p<\infty$.
Next set
$$
I_\alpha f=\chi_-^{-\alpha-2}\ast f
$$
and for fixed $\varepsilon>0$ consider the operator $
(1-d^2/dx^2)^{-(\alpha+1+\varepsilon)/2}I_\alpha$. An argument
as in \cite[Example 7.1.17, p. 167 and (3.2.9) p.72]{Ho1} shows that
$(1-d^2/dx^2)^{-(\alpha+1+\varepsilon)/2}I_\alpha f =f\ast \eta$ where ${\widehat \eta}$ is the
locally integrable function
$$
{\widehat \eta(\xi)}={-ie^{i\pi\alpha/2}\xi_+^{-(\alpha+1)} +ie^{-i\pi\alpha/2}\xi_-^{-(\alpha+1)}
\over (1+\xi^2)^{(\alpha+1+\varepsilon)/2} }.
$$
Here $\xi_+=\max(0, \xi)$  and $\xi_-=-\max(0, \xi)$.
A standard argument shows that $\eta \in L^1({\mathbb R})$.
Hence
\begin{eqnarray*}
I_{12}&\leq& C\left\| \delta_RF\ast \chi_-^{-\alpha-2} \right\|_1\le
C\left\| (1-d^2/dx^2)^{-(\alpha+1+\varepsilon)/2} \delta_RF \right\|_1
= C   \|   \delta_RF\|_{W^{ \alpha+1+\varepsilon, 1}({\mathbb R})}.
\end{eqnarray*}
 This finishes the proof.
\end{proof}

In our next results we will explain how to estimate the $L^p \to L^q$ norm of general multiplier $F(L)$ in terms of
estimate ${\rm (BR^{ \alpha}_{p, q, m})}$.

\begin{proposition}\label{prop3.3}
Suppose that $(X, d, \mu)$ satisfies the doubling property \eqref{e2.1}
and
  a non-negative self-adjoint operator $L$ acting on $L^2(X)$ satisfying condition ${\rm (BR^{ \alpha}_{p, q, m})}$
 for
 $\alpha\geq -1$ and
$1\leq p < q\leq \infty$.  Then for every $\varepsilon>0$ there exists a constant $C_\varepsilon$ such that
 for any $R>0$ and all Borel functions $F$ for which $\supp F \subset [R/2, R]$
\begin{eqnarray*}
\big\|F(\SL)P_{B(x, \rho)} \big\|_{p \to q} \leq
 CV(x,\rho)^{{1\over q}-{1\over p}}
(R\rho)^{{n}({1\over p}- {1\over q}) }   \big\| \delta_RF \big\|_{WS^{\alpha+1, 1}({\mathbb R})}
 \end{eqnarray*}
  for all
  $x\in X$ and all $\rho\geq 1/R$.
\end{proposition}

\begin{proof} We  use  the formula \eqref{e3.14} to obtain  that
  for $\alpha\geq -1,$
\begin{eqnarray*}
F(\sqrt[m] {L})
&=& \frac{1}{\Gamma(\alpha+1)}\int_0^\infty
F \ast \chi_-^{-\alpha-2} (\lambda) (\lambda-\sqrt[m]{L})_+^{\alpha} d\lambda.
\end{eqnarray*}
Since supp $F\subseteq [R/2, R]$, one can rewrite
\begin{eqnarray*}
F(\sqrt[m] {L})
   &=&
   \int_0^\infty  \lambda^{\alpha} S_{ \lambda R }^{\alpha}(\sqrt[m]{L})
   \big( \delta_RF \ast \chi_-^{-\alpha-2}\big)(\lambda)  d\lambda\nonumber\\
   &=&  \left( \int_0^2 +\int_2^{\infty} \right)
   \lambda^{\alpha} S_{ \lambda R }^{\alpha}(\sqrt[m]{L})
   \big( \delta_RF \ast \chi_-^{-\alpha-2}\big)(\lambda)  d\lambda\nonumber\\
   &=& I_1+I_2.
\end{eqnarray*}
We observe that if $\lambda\in (2, \infty)$
and $\tau\in (1/2, 1)$, then $  \chi_-^{-\alpha-2} (\lambda-\tau)=0,$
and so $   \big( \delta_RF \ast \chi_-^{-\alpha-2}\big)(\lambda)=0.$
Hence $I_2=0.$ To estimate the term $I_1$ we use
  condition ${\rm (BR^{\alpha}_{p, q, m})}$ to obtain
\begin{eqnarray*}
 \|
 F(\sqrt[m] {L}) P_{B(x, \rho)}\|_{p \to q }
&\leq&      \int_0^2
   \lambda^{\alpha} \| S_{ \lambda R }^{\alpha}(\sqrt[m]{L}) P_{B(x, \rho)}\|_{p \to q }
   \big( \delta_RF \ast \chi_-^{-\alpha-2}\big)(\lambda)  d\lambda\nonumber\\
   &\leq&   C V(x,\rho)^{{1\over q}-{1\over p}}
(R\rho)^{{n}({1\over p}- {1\over q}) }
 \int_{0}^{2} \lambda^{{n }({1\over p }-{1\over q}) +\alpha}
\big|  \delta_RF \ast \chi_-^{-\alpha-2} (\lambda)  \big| d\lambda\nonumber\\
&=&   C
V(x,\rho)^{{1\over q}-{1\over p}}
(R\rho)^{{n}({1\over p}- {1\over q}) }\left(
\int_{0}^{1/4} +\int_{1/4}^{2}   \right) \lambda^{{n }({1\over p }-{1\over q}) +\alpha}
\big|  \delta_RF   \ast \chi_-^{-\alpha-2} (\lambda)  \big| d\lambda\nonumber\\
&=&   CV(x,\rho)^{{1\over q}-{1\over p}}
(R\rho)^{{n}({1\over p}- {1\over q}) }\left(
 I_{11} +I_{12}\right).
\end{eqnarray*}
For the term $I_{11}$ we use the fact that $\supp \delta_R F\subseteq [1/2, 1]$
to obtain that if $\lambda\in (0, 1/4)$ and $\tau\in (1/2, 2)$, then $ |\chi_-^{-\alpha-2} (\lambda-\tau)|\leq C.$
This shows
\begin{eqnarray*}
I_{11}&=&  C
\int_{0}^{1/4}  \lambda^{{n }({1\over p }-{1\over q}) +\alpha}
\big|  \delta_RF  \ast \chi_-^{-\alpha-2} (\lambda)  \big| d\lambda\nonumber\\
 &\leq&C
\int_{0}^{1/4}\int_{\R} \lambda^{{n }({1\over p }-{1\over q}) +\alpha}
\big|  \chi_-^{-\alpha-2} (\lambda-\tau) \delta_RF ( \tau) \big| d\tau d\lambda\nonumber\\
 &\leq&C\| \delta_RF \|_1
\int_{0}^{1/4}  \lambda^{{n }({1\over p }-{1\over q}) +\alpha} d\lambda\nonumber\\
 &\leq&C\| \delta_RF \|_1,
\end{eqnarray*}
where we used the fact that  $n({1/p }-{1/q})  +\alpha>-1.$
Now we  estimate the term $I_{12}$, and note that
\begin{eqnarray*}
I_{12}
 &\leq&C
\int_{1/4}^{2} \lambda^{{n }({1\over p }-{1\over q}) +\alpha}
\big| \delta_RF\ast \chi_-^{-\alpha-2} (\lambda)  \big| d\lambda\nonumber\\
 &\leq&C\int_{1/32}^{8} |\delta_RF \ast \chi_-^{-\alpha-2}(\lambda)|d\lambda \le
 \| \delta_RF \|_{WS^{\alpha+1, 1}}
\end{eqnarray*}
This ends the proof.
\end{proof}

\subsection{Bochner-Riesz means imply spectral multiplier estimates}
In this section we will show that Bochner-Riesz means can be used to study spectral multipliers corresponding
to functions supported in dyadic intervals.  We assume that
   $(X, d, \mu)$ is  a  metric measure  space satisfying the doubling property
 and   $n$ is the  doubling dimension from  condition \eqref{e2.2}.

\begin{theorem}\label{th3.5}
Suppose that  there exists a constant $C>0$ such that  $V(x, \rho)\geq C\rho^n$ for all $x\in X$ and $\rho>0$.
Let $L$ be
  a non-negative self-adjoint operator $L$ acting on $L^2(X)$ satisfying Davies-Gaffney  estimates \eqref{DG}
  and condition ${\rm (G_{p_0, 2, m})}$ for some $1\leq p_0<2$. Next assume  that
  condition  ${\rm (BR^{ \alpha}_{p, q, m})}$ holds for $\alpha\geq -1$ and
$p_0< p< q< p'_0$.
Let  $p\leq r  \leq s\leq q$,
and $\beta> n(1/p-1/r) +n({1/s}-{1/q}) +\alpha+1$.
Then for  a   Borel  function $F$ such that  $\supp F\subseteq [1/4, 4]$ and
 $F \in  W^{\beta, 1}(\R)$,
 the operator
$F(t\sqrt[m]{L})$ is bounded from  $L^r(X)$ to $L^s(X)$. In addition,
\begin{equation}
\label{e3.18}    \sup_{t>0}t^{n({1\over r}-{1\over s})}\|F(t\sqrt[m]{L})\|_{r\to s} \leq
C
\|F\|_{W^{\beta, 1}(\R)}.
\end{equation}
\end{theorem}

\begin{proof}
Let   $\phi \in C_c^{\infty}(\mathbb R) $   be  a function such that $\supp \phi\subseteq \{ \xi: 1/4\leq |\xi|\leq 1\}$ and
 $
\sum_{\ell\in \ZZ} \phi(2^{-\ell} \lambda)=1$ for all
${\lambda>0}.
$
Set $\phi_0(\lambda)= 1-\sum_{\ell=1}^{\infty} \phi(2^{-\ell} \lambda)$,
\begin{eqnarray}\label{e3.19}
G^{(0)}(\lambda)=\frac{1}{2\pi}\int_{-\infty}^{+\infty}
 \phi_0(\tau) \hat{G}(\tau) e^{i\tau\lambda} \;d\tau
\end{eqnarray}
and
\begin{eqnarray}\label{e3.20}
G^{(\ell)}(\lambda) =\frac{1}{2\pi}\int_{-\infty}^{+\infty}
 \phi(2^{-\ell}\tau) \hat{G}(\tau) e^{i\tau\lambda} \;d\tau,
\end{eqnarray}
where $G(\lambda)=F(\sqrt[m]\lambda) e^{\lambda}$.  Note that   by the Fourier inversion formula
$$
G(\lambda)=\sum_{\ell=0}^{\infty}G^{(\ell)}(\lambda).
$$
Then
\begin{eqnarray}\label{e3.21}
F(\sqrt[m]\lambda)= G(\lambda)e^{-\lambda}
&=&\sum_{\ell=0}^{\infty}G^{(\ell)}(\lambda)e^{-\lambda}
=:\sum_{\ell=0}^{\infty}F^{(\ell)}(\sqrt[m]\lambda)
\end{eqnarray}
so for any $f\in L^r(X)$
\begin{eqnarray}\label{e3.22}
\big\|F(t\SL)f\big\|_{s } &\le&
\sum_{\ell=0}^{\infty}\big\|F^{(\ell)}(t\sqrt[m]{L}) f\big\|_{s}, \ \ \ \ \ p\leq r  \leq s\leq q.
\end{eqnarray}

Next we   fix  $ \varepsilon>0$  such that
\begin{equation}\label{epsilon}
2n\varepsilon [({1/s}-{1/r})+ ({1/p}-{1/q})]\leq  \beta-n({1/s}-{1/r})-n({1/p}-{1/q})-\alpha-1.
\end{equation}
For every $t>0$ and every  $\ell$ set $\rho_{\ell}=2^{\ell(1+\varepsilon)}t$.
Then  we choose a sequence $(x_n)  \in X$ such that
$d(x_i,x_j)> \rho_{\ell}/10$ for $i\neq j$ and $\sup_{x\in X}\inf_i d(x,x_i)
\le \rho_{\ell}/10$. Such sequence exists because $X$ is separable.
Now set $B_i=B(x_i, \rho_{\ell})$ and define $\widetilde{B_i}$ by the formula
$$\widetilde{B_i}=\bar{B}\left(x_i,\frac{\rho_{\ell}}{10}\right)\setminus
\bigcup_{j<i}\bar{B}\left(x_j,\frac{\rho_{\ell}}{10}\right),$$
where $\bar{B}\left(x, \rho_{\ell}\right)=\{y\in X \colon d(x,y)
\le \rho_{\ell}\}$. Note that for $i\neq j$,
 $B(x_i, \frac{\rho_{\ell}}{20}) \cap B(x_j, \frac{\rho_{\ell}}{20})=\emptyset$.

Observe that for every $k\in{\mathbb N}$,
\begin{eqnarray}\label{e3.23}
\sup_i\#\{j:\;d(x_i,x_j)\le  2^k\rho_{\ell}\} &\le&
  \sup_{d(x,y) \le 2^k\rho_{\ell}}  {V(x, 2^{k+1}\rho_{\ell})\over
    V(y, \frac{\rho_{\ell}}{20})}\nonumber\\
   &\le&
  \sup_y  {V(y, 2^{k+2}\rho_{\ell})\over
  V(y, \frac{\rho_{\ell}}{20})}
    \le C 2^{kn}.
\end{eqnarray}
Set
$
\D_{\rho_{\ell}}=\{ (x,\, y)\in X\times X: {d}(x,\, y) \le \rho_{\ell} \}.
$
It is not difficult to see that
\begin{eqnarray}\label{e3.24}
\D_{\rho_{\ell}}  \subseteq \bigcup_{\{i,j:\, d(x_i,x_j)<
 2 \rho_{\ell}\}} \widetilde{B}_i\times \widetilde{B}_j \subseteq \D_{4 \rho_{\ell}}.
\end{eqnarray}

Now let  $\psi \in C_c^\infty(1/16, 4)$ be a function such
that $\psi(\lambda)=1$ for $\lambda \in (1/8,3)$,
and  we decompose
\begin{eqnarray}\label{e3.25}
 F^{(\ell)}(t\sqrt[m]{L})f
&=& \sum_{i,j:\, {d}(x_i,x_j)< 2\rho_{\ell}} P_{\widetilde B_i}\big[\psi F^{(\ell)} (t\sqrt[m]{L}) \big]P_{\widetilde
B_j}f\nonumber\\
  &&+ \sum_{i,j:\, {d}(x_i,x_j)< 2\rho_{\ell}} P_{\widetilde B_i}\big[(1-\psi)
F^{(\ell)} (t\sqrt[m]{L}) \big] P_{\widetilde
B_j}f\nonumber\\
&&+  \sum_{i,j:\, {d}(x_i,x_j)\geq 2\rho_{\ell}} P_{\widetilde B_i}F^{(\ell)}(t\sqrt[m]{L}) P_{\widetilde
B_j}f=I+ I\!I +I\!I\!I.
\end{eqnarray}

\medskip

\noindent
{\underline{\em Estimate for  {\it I}}}.   \ Note that $p\leq r  \leq s\leq q$.
By
 H\"older's inequality
\begin{eqnarray}\label{e3.26}
 \| \sum_{i,j:\, {d}(x_i,x_j)< 2\rho_{\ell}} P_{\widetilde B_i}\big(\psi F^{(\ell)}\big) (t\sqrt[m]{L}) P_{\widetilde
B_j}f\|_{s}^s
 =\sum_i \|\sum_{j:\,{d}(x_i,x_j)<
2\rho_{\ell}} P_{\widetilde B_i}\big(\psi F^{(\ell)}\big)(t\sqrt[m]{L})P_{\widetilde B_j}f\|_{s}^s  \nonumber \\
\le C   \sum_i \sum_{j:\,{d}(x_i,x_j)< 2\rho_{\ell}} \| P_{\widetilde
B_i}\big(\psi F^{(\ell)}\big)(t\sqrt[m]{L})P_{\widetilde B_j}f\|_{s}^s
\nonumber \\
\le C  \sum_i   \sum_{j:\,{d}(x_i,x_j)< 2\rho_{\ell}}
V(\widetilde{B}_i)^{s({1\over s}-{1\over q})} \|P_{\widetilde
B_i}\big(\psi F^{(\ell)}\big)(t\sqrt[m]{L})P_{\widetilde B_j}f\|_{q}^s
\nonumber \\ \le C   \sum_j     V( {B}_j)^{s({1\over s}-{1\over
q})} \|\big(\psi F^{(\ell)}\big)(t\sqrt[m]{L})P_{ \widetilde B_j}f\|_{q}^s
\nonumber \\
\le C  \sum_j V( {B}_j)^{s({1\over s}-{1\over
q})}\|\big(\psi F^{(\ell)}\big)(t\sqrt[m]{L})P_{\widetilde B_j}\|_{r\to q}^s
\|P_{\widetilde B_j}f\|_{r}^s
\nonumber \\
 \le C \sup_{x\in X}\big\{V(x,\rho_{\ell})^{s({1\over s}-{1\over
q})} V(x,\rho_{\ell})^{s({1\over p}-{1\over
r})}
\|\big(\psi F^{(\ell)}\big)(t\sqrt[m]{L}) P_{\widetilde B_j}\|^s_{p\to q}\big\}  \sum_j\|P_{\widetilde
B_j}f\|_{r}^s
\nonumber \\
=  C   \sup_{x\in X}\big\{V(x,\rho_{\ell})^{s({1\over s}-{1\over
r})}  V(x,\rho_{\ell})^{s({1\over p}-{1\over
q})} \big\} \|\big(\psi F^{(\ell)}\big)(t\sqrt[m]{L}) P_{\widetilde B_j}\|^s_{p\to q}\big\}  \|f\|_{r}^s.
\end{eqnarray}

In Proposition~\ref{prop3.3} we assume
$\supp F \subset [R/2, R]$. To adjust the multipliers which we consider here
to this requirement
we write
$
 \big(\psi F^{(\ell)}\big)(t\sqrt[m]{L})
=\sum_{k=0}^{7}  \big(\chi_{[2^{k-4}, 2^{k-3}) }\psi
F^{(\ell)}\big)(t\sqrt[m]{L})
$
Now by  Proposition~\ref{prop3.3} for every  $\ell \ge 4$,
  \begin{eqnarray*}
&&\hspace{-0.8cm} \big\|\big(\psi F^{(\ell)}\big)(t\sqrt[m]{L})P_{B(x, \rho_\ell)}\big\|_{p \to q}
 \\
&\leq&  C\sum_{k=0}^{7}  V(x,\rho_\ell)^{{1\over q}-{1\over p}}2^{\ell(1+\varepsilon)n({1\over
p}-{1\over q})} \big\| \delta_{2^{k-3}t^{-1}}\big(\psi F^{\ell}\big)(t\cdot) \big\|_{WS^{\alpha+1, 1} }
 \nonumber
 \end{eqnarray*}
 Next we note that it follows from \eqref{e3.20}
 that
 $$
 \big\| G^{(\ell)} \big\|_{WS^{\alpha+1, 1} }
 \leq  C2^{(\alpha+1)\ell} \|G^{(\ell)}\|_1
 $$
Hence  for $k=0, 1, \ldots, 7,$
  \begin{eqnarray*}
\big\| \delta_{2^{k-3}t^{-1}}\big(\psi F^{\ell}\big)(t\cdot) \big\|_{WS^{\alpha+1, 1} }
&\leq&C  \big\| G^{(\ell)} \big\|_{WS^{\alpha+1, 1} }
 \leq  C2^{(\alpha+1)\ell} \|G^{(\ell)}\|_1
 \end{eqnarray*}
This gives
 \begin{eqnarray*}
  \big\|\big(\psi F^{(\ell)}\big)(t\sqrt[m]{L})P_{B(x, \rho_\ell)}\big\|_{p\to q}
&\leq&    V(x,\rho_\ell)^{{1\over q}-{1\over p}}2^{\ell(1+\varepsilon)n({1\over
p}-{1\over q})} 2^{( \alpha +1)\ell} \|G^{(\ell)}\|_1 \nonumber
 \end{eqnarray*}
for every $\ell\geq 4.$ On the other hand, for
 $\ell=0, 1,2,3$,  we note that   by Proposition 2.3 of \cite{SYY},
  ${\rm (G_{p_0,2,m})} \Rightarrow {\rm (ST^{\infty}_{p_0, 2, m})}$,
  and thus
 $\big\|\big(\psi F^{(\ell)}\big)(t\sqrt[m]{L})P_{B(x, \rho_\ell)}\big\|_{p_0\to 2 }
 \leq
   CV(x,\rho_\ell)^{{1\over 2}-{1\over p_0}}  \|F  \|_{1}.
   $
Since  $V(x, \rho)\geq C\rho^n$ for all $x\in X$ and $\rho>0$,
we have that
 \begin{eqnarray}\label{e3.27}
&&\hspace{-1.5cm}
\sum_{\ell=1}^{\infty}
 \| \sum_{i,j:\, {d}(x_i,x_j)< 2\rho_{\ell}} P_{\widetilde B_i}\big(\psi F^{(\ell)}\big) (t\sqrt[m]{L}) P_{\widetilde
B_j}f\|_{s}\nonumber\\
 &\le &
\sum_{\ell=1}^{3} \sup_{x\in X}\big\{ V(x,\rho_{\ell})^{  {1\over s}-{1\over
r} }  V(x,\rho_{\ell})^{{1\over p }-{1\over
2}}\big\|\big(\psi F^{(\ell)}\big)(t\sqrt[m]{L})P_{B(x,\rho_{\ell})}\big\|_{p_0\to
2} \big\}\nonumber\\
 &+&
\sum_{\ell=4}^{\infty} \sup_{x\in X}\big\{ V(x,\rho_{\ell})^{  {1\over s}-{1\over
r} }  V(x,\rho_{\ell})^{{1\over p }-{1\over
q}}\big\|\big(\psi F^{(\ell)}\big)(t\sqrt[m]{L})P_{B(x,\rho_{\ell})}\big\|_{p \to
q} \big\}\nonumber\\
 &\le & C t^{ n({1\over s}-{1\over r})}
 \Big(\|F  \|_{1} +   \sum_{\ell=4}^{\infty}  2^{\ell(1+\varepsilon) n({1\over s}-{1\over r})}
  2^{\ell(1+\varepsilon) n({1\over p }-{1\over q})}2^{( \alpha +1)\ell} \|G^{(\ell)}\|_1\Big)
\nonumber\\
 &\leq &  C t^{ n({1\over s}-{1\over r})}
  \|G\|_{B_{1,\, 1}^{\gamma }},
\end{eqnarray}
where  $\gamma=n(1/p-1/r)+n(1/s -1/q) +\alpha+1+\delta$ and $\delta={\varepsilon n }({1/p}-{1/r})
 +{\varepsilon n }({1/s }-{1/q})$.
The last inequality follows from definition of Besov space.
See e.g. \cite[Chap.~VI ]{BL}.
By \eqref{epsilon}
 $$W^{\beta, 1}\subseteq
B_{1, \, 1}^{ \gamma} \quad \mbox{ with} \quad
 \|G\|_{B_{1,\, 1}^{ \gamma}}\le C_\alpha
\|G\|_{W^{\beta, 1}}$$
where $\gamma=n(1/p-1/r)+n(1/s -1/q) +\alpha+1+\delta$,  see again
\cite{BL}.  However, $\supp F\subseteq [1/4, 4]$ so $ \|G\|_{W^{\beta, 1}}\le
\|F\|_{W^{\beta, 1}}$. Hence the forgoing estimates give
\begin{eqnarray}\label{e3.28}
{\rm LHS \ of \ \eqref{e3.27}} \le  Ct^{ n({1\over s}-{1\over r})}\|F\|_{W^{\beta, 1}}.
\end{eqnarray}

\noindent
{\underline{\em Estimate of   $I\!I$}}.  \
Repeat an argument  leading up to  \eqref{e3.26}, it is easy to see that
\begin{eqnarray*}
 \|\sum_{i,j:\, {d}(x_i,x_j)< 2\rho_{\ell}} P_{\widetilde B_i}\big((1-\psi)
F^{(\ell)}\big) (t\sqrt[m]{L}) P_{\widetilde
B_j}f \|_{  s}
&\leq&C\sup_{x\in X} \|\big((1-\psi) F^{(\ell)}\big)(t\sqrt[m]{L})P_{B(x,\rho_{\ell})}\|_{r\to s}\|f\|_r  \\
&\leq &C\|\big((1-\psi) F^{(\ell)}\big)(t\sqrt[m]{L}) \|_{r\to s} \|f\|_r,
\end{eqnarray*}
where, for a fixed $N$,  one has the uniform estimates
 \begin{eqnarray*}
\Big|\Big({d\over d\lambda}\Big)^{\kappa} \big((1-\psi) F^{(\ell)}\big)(\lambda)\Big|\leq C_{\kappa}2^{-\ell N}
 (1+|\lambda|)^{-N}\|F\|_{L^1(\R)}.
\end{eqnarray*}
But   Proposition~\ref{prop2.1} then implies that for every $p\leq r\leq s\leq q$,
 \begin{eqnarray*}
  \|\big((1-\psi) F^{(\ell)}\big)(t\sqrt[m]{L}) \|_{r\to s}
  &\leq&   \|\big((1-\psi) F^{(\ell)}\big)(t\sqrt[m]{L}) (1+t\sqrt[m]L)^M\|_{s\to s}
 \|  (1+t\sqrt[m]L)^{-M}\|_{r\to s}\nonumber\\
   &\leq& C  2^{-\ell N}t^{{n }({1\over r}-{1\over s})}\|F\|_{L^1(\R)},
\end{eqnarray*}
which gives
  \begin{eqnarray}\label{e3.29}
\sum_{\ell=0}^{\infty}
\|\big((1-\psi) F^{(\ell)}\big)(t\sqrt[m]{L}) \|_{r\to s} \leq Ct^{{n }({1\over r}-{1\over s})}  \|F\|_{L^1(\R)}.
\end{eqnarray}

\medskip

\noindent
{\underline {\em Estimate of  $I\!I\!I$}}.  \ Note that
\begin{eqnarray*}
\big\|\sum_{i,j:\, {d}(x_i,x_j)\geq 2^{\ell(1+\varepsilon)} t}
P_{\widetilde B_i}F^{(\ell)}(t\sqrt[m]{L}) P_{\widetilde
B_j} f \big\|_s^s&=& \sum_{i} \big\|\sum_{j:\, {d}(x_i,x_j)\geq 2^{\ell(1+\varepsilon)} t}
P_{\widetilde B_i}F^{(\ell)}(t\sqrt[m]{L}) P_{\widetilde
B_j} f \big\|_s^s\\
&\leq& \sum_{i}\Big( \sum_{j:\, {d}(x_i,x_j)\geq 2^{\ell(1+\varepsilon)} t} \big\|
P_{\widetilde B_i}F^{(\ell)}(t\sqrt[m]{L}) P_{\widetilde
B_j} f \big\|_s \Big)^s.
\end{eqnarray*}
 Recall that  $G(\lambda)=F(\sqrt[m]\lambda) e^{\lambda}$.
 By the formula  \eqref{e3.20}, it follows from an argument as in \cite[Lemma 4.3]{SYY} that
 For
all  $\ell=0, 1, 2, \ldots$  and  all $x_i, x_j$ with  ${d}(x_i,x_j)\geq 2^{\ell(1+\varepsilon)} t$,   there exist
some positive constants $C, c_1, c_2>0$ such that for $p_0 <r\leq s<p'_0$,
\begin{eqnarray}
\label{e3.30}
&&\hspace{-1.5cm}
 \big\| P_{\widetilde B_i}F^{(\ell)}(t\sqrt[m]{L}) P_{\widetilde
 B_j} f\big\|_s\nonumber\\
   &\leq&  C
  \|P_{\widetilde B_j} f\|_r \int_{-\infty}^{+\infty}
 |\phi(2^{-\ell}\tau) \hat{G}(\tau)|
 \big\| P_{\widetilde B_i}e^{(i\tau-1)t^mL}  P_{\widetilde B_j}  \big\|_{r\to s}
  \;d\tau, \nonumber\\
  &\leq&
C  t^{n({1\over r}-{1\over s})} e^{-c_1 2^{ \varepsilon \ell m\over m-1}  }
\exp\Big(- c_2\Big({d(x_i, x_j)\over  2^{\ell} t}\Big)^{m\over m-1}\Big) \|F\|_{1} \|P_{\widetilde B_j} f\|_r.
\end{eqnarray}
which, together with   the Cauchy-Schwarz inequality, yields
 \begin{eqnarray*}
&&\hspace{-1cm}\big\|\sum_{i,j:\, {d}(x_i,x_j)\geq 2^{\ell(1+\varepsilon)} t}
P_{\widetilde B_i}F^{(\ell)}(t\sqrt[m]{L}) P_{\widetilde
B_j} f \big\|_s^s\\
&\leq & Ct^{{ns }({1\over r}-{1\over s})}e^{- c_1s 2^{\varepsilon\ell  m\over m-1}  } \|F\|_{1}^s
\sum_{i} \Big\{ \sum_{j:\, {d}(x_i,x_j)\geq 2^{\ell(1+\varepsilon)} t}
 \exp\Big(- c_2\Big({d(x_i, x_j)\over  2^{\ell} t}\Big)^{m\over m-1}\Big)
  \|P_{\widetilde B_j} f\|_r\Big\}^s \\
	&\leq&  Ct^{{n s}({1\over r}-{1\over s})}e^{- c_1 s 2^{\varepsilon\ell  m\over m-1}  } \|F\|_{1}^s
	 \left(\sum_{ j}  \|P_{\widetilde B_j} f\|_r^r\right)^{s/r}
	 \sum_{i:\, {d}(x_i,x_j)\geq 2^{\ell(1+\varepsilon)} t}
\exp\Big(-c_2\Big({d(x_i, x_j)\over 2^{\ell} t}\Big)^{m\over m-1}\Big)  \\
	&\leq&  Ct^{{ns }({1\over r}-{1\over s})}e^{- c_1s  2^{\varepsilon\ell  m\over m-1}  } \|F\|_{1}^s  \| f\|_r^s.
\end{eqnarray*}
 Therefore,
 \begin{eqnarray}\label{e3.31}
 \sum_{\ell=0}^{\infty}    \big\|  \sum_{i,j:\, {d}(x_i,x_j)\geq 2^{\ell(1+\varepsilon)} t}
P_{\widetilde B_i}F^{(\ell)}(t\sqrt[m]{L}) P_{\widetilde
B_j} f\big\|_s &\leq& C t^{{n }({1\over r}-{1\over s})}\sum_{\ell=0}^{\infty} e^{- c_1
 2^{\varepsilon\ell m\over m-1}  } \|F\|_{ 1} \|f\|_r\nonumber\\
&\leq& C t^{{n }({1\over r}-{1\over s})} \|F\|_{ 1} \|f\|_r.
\end{eqnarray}
Estimate \eqref{e3.18}  then follows  from  \eqref{e3.22},  \eqref{e3.25},
\eqref{e3.27},   \eqref{e3.28}, \eqref{e3.29}  and \eqref{e3.31}.
This completes the proof of Theorem~\ref{th3.5}.
\end{proof}

\smallskip

\begin{remark}\label{re3.6} From the proof of Theorem~\ref{th3.5}, we can see that the result of
$$\sup_{t>0} \|F(t\sqrt[m]{L})\|_{r\to r} \leq
C
$$
in Theorem~\ref{th3.5} (i.e., $r=s$ in \eqref{e3.18})
holds under
the assumption that $(X, d, \mu)$ satisfies the doubling condition \eqref{e2.1} only. In this case,
we do not need the assumption that  $V(x, \rho)\geq C\rho^n$ for all $x\in X$ and $\rho>0$.
See also Theorem 4.2, \cite{SYY}.
\end{remark}

The following corollary is a consequence of Theorem~\ref{th3.5}.

\begin{coro}\label{coro3.6}
Suppose that  there exists a constant $C>0$ such that  $V(x, \rho)\geq C\rho^n$ for all $x\in X$ and $\rho>0$.
Next assume that
  a non-negative self-adjoint operator $L$ acting on $L^2(X)$ satisfies  estimates \eqref{DG}
  and ${\rm (G_{p_0, 2, m})}$   for some $1\leq p_0< 2.$
Then   ${\rm (BR^{ \alpha}_{p, q, m})}$  for  $\alpha\geq -1$ and $ p_0< p<q< p'_0$,
implies
  $$\left\|  S^{\delta}_{R^m}(L)\right\|_{r\to s}  \leq C
R^{{n}({1\over r}- {1\over s}) } $$
for all  $ p\leq r\leq s \leq q$  and ${\rm Re}\, \delta>\alpha +n({1/p}-{1/r})
+n({1/s }-{1/q}).$

In particular, if
   ${\rm (BR^{ \alpha}_{p_\alpha, q_\alpha, m})}$    holds for
\begin{eqnarray}\label{e3.32}
\left({1\over p_\alpha}, {1\over q_\alpha}\right)
=  \left({n+1+2\alpha\over 2n}-{2\alpha\over n+1}, {n+1+2\alpha\over 2n}\right),
\end{eqnarray}
then
  $$\left\|S^{\delta}_{R}(L)\right\|_{r\to r} \leq C
$$
for all $p_{\alpha}\leq r\leq p_{\alpha}'$
and $\delta>{n }({1/p_{\alpha}}-{1/2})-{1/2}$.
\end{coro}

\begin{proof} Let $F(\lambda) = (1-\lambda^m)^{\delta}_+$ and  $\delta=\sigma+i\tau$.  We set
$$
 F(\lambda) =F(\lambda)  \phi(\lambda^m) + F(\lambda)(1-\phi(\lambda^m))=:F_1(\lambda^m) +F_2(\lambda^m),
$$
 where   $\phi\in C^{\infty}(\R)$ is supported in  $ \{\lambda:  |\lambda| \geq 1/4 \}$
 and $\phi =1$ for all $|\lambda|\geq 1/2$.
  It is known   that
if $0<s<\sigma+ 1 $,
   then $(1-|\lambda|^m)_+^\delta\in  W^{s,1}(\mathbb{R})$ with
$
\big\| (1-|\lambda|^m)_+^\delta\big\|_{W^{s, 1}(\mathbb{R})}\leq C\, e^{c|\tau|}
$
for constants $C,c>0$ independent of   $s$, see for example  \cite[Lemma 4.4]{BGSY}.
  This, in combination with   Theorem~\ref{th3.5},
shows that
$$ \sup_{R>0} R^{n(\frac{1}{s}-\frac{1}{r})} \left\|F_2\left({L\over R^m}\right)\right\|_{r\to s}\leq C$$
  for all for   $ p \leq r\leq s \leq q$  and $\sigma>\alpha +n({1/p}-{1/r})
+n({1/s}-{1/q}).$
 On the other hand,
we note that  $V(x, \rho)\geq C\rho^n$ for all $x\in X$ and $\rho>0$.
 by Propositions  \ref{prop2.1} and  \ref{prop2.2}  that for $p_0 <r<s<p'_0$ and
 for $N>m(1/p-1/2)$,
\begin{eqnarray*}
\|F_1\left({L\over R^m}\right)\|_{r\to s}&=&
\left\|F_1\left({L\over R^m}\right)\left(1+{  L\over R^m}\right)^{N/m}\right\|_{r\to r}
\left\| \left(1+ {L\over R^m}\right)^{-N/m} \right\|_{r\to s}\nonumber\\
&\leq&   C R^{{n}({1\over r}-{1\over s})}.
\end{eqnarray*}
This proves  ${\rm (BR^{ \delta}_{r, s, m})}$.

Now  assume \eqref{e3.32}. It follows  that
 $
 \|S_R^{\delta}(\sqrt[m]L)\|_{ p_{\alpha}\to
p_\alpha}\leq C
 $
 for ${\rm Re}\, \delta>\alpha +n({1/p_{\alpha}}-{1/q_\alpha})
={n }({1/p_{\alpha}}-{1/2})-{1/2}.$
By  duality and interpolation, ${\rm (BR^{ \delta}_{r, r, m})}$  holds for all $p_{\alpha}\leq r\leq p_{\alpha}'$
and $\delta>{n }({1/p_{\alpha}}-{1/2})-{1/2}$.
 The proof is  complete.
\end{proof}

\subsection{Estimates for the Bochner-Riesz means with negative index}\
In previous section in Corollary \ref{coro3.6} we prove   that  one can narrow the
gap between $p$ and $q$ in condition
${\rm (BR^{ \alpha}_{p, q, m})}$ by increasing the order of Bochner-Riesz means $\alpha$.
In this section we describe relation between various  ${\rm (BR^{ \alpha}_{p, q, m})}$ of a different nature.
 This time the argument is based on $T^*T$ type argument and Stein's complex
interpolation. The heart of the matter in our discussion is the fact that Stein-Tomas restriction estimate
is essentially equivalent
with the full description of $L^p \to  L^q$ mapping properties of Bochner-Riesz means of order
$1/2$.

Given some $1\leq p< 2$, we set
\begin{eqnarray*}
A=\left( 1, {n+1+2\alpha\over 2n}\right),&& A'=\left({n-1-2\alpha\over 2n}, 0\right),\nonumber\\
B(p)=\left({n+1+2\alpha\over 2n}+\alpha-{2\alpha\over p}, {n+1+2\alpha\over 2n}\right),&&
B'(p)=\left({n-1-2\alpha\over 2n},  {2\alpha\over p}-\alpha+{n-1-2\alpha\over 2n}\right),\nonumber\\
C(p)=\left({1 \over p}, {n+1+2\alpha\over 2n}\right),&&
C'(p)=\left({n-1-2\alpha\over 2n},   1-{1 \over p}\right),\\
D(p)=\left({1\over 2} +{ \alpha }-{2 \alpha\over  p }, {1 \over 2 }\right),&&
D'(p)=\left({1 \over 2 }, {1\over 2}-{  \alpha }+{2 \alpha\over  p }\right).\nonumber
\end{eqnarray*}
Denote by
$\Delta_{\alpha}(p, n)$  the open pentagon with vertices $A, B(p), B'(p), A'$ and $(1,0)$.
 Namely,
  \begin{eqnarray*}
  \hspace{1cm}
 \Delta_{\alpha}(p, n)=\left\{ \left({1\over r}, {1\over s}\right)\in (0, 1)\times (0,1):\ \
 \min\Big(\frac{1}{r}-\frac{1}{2},\
 \frac{1}{2}-\frac{1}{s}\Big)> -{2\alpha+1\over 2n}, \ \ \alpha- \frac{2\alpha}{p}< \frac{1}{r}-\frac{1}{s}
  \right\}.
 \end{eqnarray*}
We are now in position to state our next result.

\begin{theorem}\label{th3.7}
Suppose that  there exists a constant $C>0$ such that
$C^{-1}\rho^n\leq V(x, \rho)\leq C\rho^n$ for all $x\in X$ and $\rho>0$.
Next assume that
  a non-negative self-adjoint operator $L$ acting on $L^2(X)$ satisfies  estimates \eqref{DG}
  and ${\rm (G_{p_0, 2, m})}$ for some $1\leq p_0<2$
  and   ${\rm (BR^{-1}_{p, p', m})}$  holds for some $p_0< p< 2$.
 $$\left\| S^{\alpha}_{R^m}(L)\right\|_{r\to s} \leq C
R^{{n}({1\over r}- {1\over s}) } $$
 if each of the following conditions holds:
\begin{itemize}
\item[(1)]   $  \alpha>  n(1/p-1/2)-1/2$,    $ p_0<r\leq s<p_0'$, $ r <q_\alpha$ and $q_\alpha' < s$ where
$q_\alpha=\max\{1,\frac{2n}{n+1+2\alpha}\}$.

\item[(2)]   $ n(1/p-1/2)-1/2 \ge  \alpha> 0$, $p_0<r\leq s<p_0'$,   $(1/r, 1/s)\in \Delta_\alpha(p,n)$
 and $(1/r, 1/s)$ is strictly below the lines joining  the point $(1/2,1/2)$ to $C(p)$  and $C'(p)$.

\item[(3)]   $ -1/2< \alpha\leq 0$, $p_0<r\leq s<p_0'$,   $(1/r, 1/s)\in \Delta_\alpha(p,n)$ and
  $(1/r, 1/s)$ is strictly below  the lines joining $D(p)$ to $C(p)$; $D(p)$ to $D'(p)$ and $D'(p)$ to $C'(p)$.

\item[(4)]      $-1<\alpha\leq  -1/2 $,     $ p_0<r\leq s<p_0'$,  $ \alpha- \frac{2\alpha}{p}< \frac{1}{r}-\frac{1}{s}$,
 $r<q_\alpha'$ and
 $q_\alpha <s$,
 where
$
1/q_\alpha=
 1+\alpha -{(2\alpha+1)/p}.
$
\end{itemize}
\end{theorem}

The proof of Theorem~\ref{th3.7} is based on the following interpolation result.

\begin{lemma}\label{le3.8}
Suppose that  there exists a constant $C>0$ such that  $V(x, r)\geq Cr^n$ for all $x\in X$ and $r>0$.
Next assume that
  a non-negative self-adjoint operator $L$ acting on $L^2(X)$ satisfies  estimates \eqref{DG}
  and ${\rm (G_{p_0, 2, m})}$ for some $1\leq p_0<2$
  and ${\rm (BR^{\delta_i}_{p_i, q_i, m})}$  holds
 for some $ \delta_i, p_i, q_i, i=1,2  $ such that
$p_0< p_i\leq q_i <p_0'$
and $ -\frac{2\delta_i}{n+1}\leq \Big(\frac{1}{p_i}-\frac{1}{q_i}\Big)$.
Then for every
$\theta\in (0, 1)$,
$$
\|S_\lambda^{\alpha } (\sqrt[m]L)  \|_{ p_{\theta}\to
q_{\theta}}
\leq C\lambda^{{n }({1\over p_{\theta}}-{1\over q_{\theta}})}, \ \ \ \ \ \lambda>0
$$
holds for   $\alpha> \delta_{\theta}=\theta \delta_1 +(1-\theta)\delta_2$ and
$$
{1\over p_{\theta}}={\theta\over p_1}  + {1-\theta\over p_2},\ \ \
{1\over q_{\theta}}={\theta\over q_1}  + {1-\theta\over q_2}.
 $$
\end{lemma}

\begin{proof}
By Lemma~\ref{le3.2} and Corollary~\ref{coro3.6} for any ${\rm Re}\,  \delta>\delta_i$
$$
\|S_\lambda^{\delta } (\sqrt[m]L)  \|_{ p_i\to
q_i}
\leq C\lambda^{{n }({1\over p_i}-{1\over q_i})}, \ \ \ \ \ \lambda>0
$$
for   $i=1, 2$. The proof then follows from Stein's classical complex  interpolation
theorem \cite{St1} for analytic families of operators.
\end{proof}

\medskip

\noindent
{\bf Proof of Theorem~\ref{th3.7}.}
We   first show that for all $ n(1/p-1/2)-1/2 \ge  \alpha> -1/2$ and every  $\varepsilon>0,$
 ${\rm (BR^{\alpha+\varepsilon}_{r, s, m})}$  holds  for $(1/r, 1/s)=C(p)=\left({1\over p}, {n+1+2\alpha\over 2n}\right).$

Indeed we assume that $L$ satisfies condition ${\rm (BR^{-1}_{p, p', m})}$
for some $p_0<p< 2$. By  Lemma~\ref{le3.2} estimate ${\rm (BR^{-1}_{r, s, m})}$
holds for all $p_0<r\leq p\leq p'\leq s<p'_0$.  Next by Corollary~\ref{coro3.6}
\begin{eqnarray*}
\|S_\lambda^{-1+ 2\varepsilon } (\sqrt[m]L)  \|_{ p\to
p'}\leq C \lambda^{ {n }({1\over p}-{1\over p'})}, \ \ \ \  \ \lambda>0
\end{eqnarray*}
for all $\varepsilon>0.$
By $TT^{\ast}$ argument,
\begin{eqnarray}\label{e3.35}
 \|S_\lambda^{-1/2+  \varepsilon} (\sqrt[m]L)   \|^2_{p\to 2}
  = \|S_\lambda^{-1+ 2\varepsilon} (\sqrt[m]L) \|_{p\to p'}\leq C \lambda^{ 2{n }({1\over p}-{1\over 2})}.
\end{eqnarray}
Now by Corollary~\ref{coro3.6}  for every $\varepsilon>0$,
\begin{eqnarray*}
\|S_\lambda^{\alpha+\varepsilon} (\sqrt[m]L) \|_{p\to q}
 \leq C \lambda^{ {n }({1\over p}-{1\over q})},
\end{eqnarray*}
where $q= \frac{2n}{n+1+2\alpha}$. This proves estimate ${\rm (BR^{\alpha+\varepsilon}_{r, s, m})}$
  for $(1/r, 1/s)=C(p)=\left({1\over p}, {n+1+2\alpha\over 2n}\right).$

Now point (2) follows from the above observation, straightforward $L^2$ estimates for
$\alpha=0$ (that is $\|S_\lambda^{0} (\sqrt[m]L)  \|_{ 2\to2}\leq 1$) ,  duality and Lemmas \ref{le3.2}
and \ref{le3.8}. The proof of point (1) is simple adjustment of the above argument based on
the fact that in virtue of Lemma \ref{lemma 3.6}  for any $p_0< p_1< p_2< 2$ condition  ${\rm (BR^{-1}_{p_2, p_2', m})}$
implies ${\rm (BR^{-1}_{p_1, p_1', m})}$.

 \medskip
Interpolation using Lemma \ref{le3.8} between \eqref{e3.35} and $L^2$ estimates
$\|S_\lambda^{0} (\sqrt[m]L)  \|_{ 2\to2}\leq 1$ yields
\begin{eqnarray}\label{e3.38}
\|S_{\lambda}^{\alpha+\varepsilon}(\sqrt[m]L)f\|_2
&\leq & C \lambda^{n({1\over r}-{1\over 2})}\|f\|_r
\end{eqnarray}
for   $ {1\over r}= {1\over 2}
 +{2 \alpha\over 2}-{2 \alpha\over  p }$ and $-{1\over 2}\leq \alpha\leq 0,$
which means that  ${\rm (BR^{ \alpha_\varepsilon}_{r, s, m})}$ holds for $(1/r, 1/s)=D(p)
=\left({1\over 2} +{2 \alpha\over 2}-{2 \alpha\over  p } , {1 \over 2 }\right).$

Now point (3) is a consequence of estimates  ${\rm (BR^{ \alpha+\varepsilon}_{r, s, m})}$
for $(1/r, 1/s)=D(p)$ and  $(1/r, 1/s)=C(p)$, the Riesz-Thorin theorem (or Lemma \ref{le3.8}), duality and Lemma~\ref{le3.2}.

  \smallskip

Since  ${\rm (BR^{ \alpha+\varepsilon}_{r, s, m})}$ holds  for  $(1/r, 1/s)=C(p)$
 and $\alpha =-1/2$, we can apply
the argument similar to the discussion described above and  assumption ${\rm (BR^{-1}_{p, p', m})}$ to obtain
 point (4). The proof of Theorem~\ref{th3.7} is end.
 \hfill{} $\Box$

\smallskip

\begin{remark}\label{re3.62} It follows from Lemma~\ref{le3.2} that if  the operator $L$ satisfies the
Gaussian upper bounds \eqref{GE}
and condition ${\rm (BR^{-1}_{p, p', m})}$  for some $1\leq p< 2$,
then  restriction $ p_0<r\leq s<p_0'$ can be removed from all points (1)-(4) and set $\Delta_{\alpha}(p, n)$
can be replaced by
\begin{eqnarray*}
\hspace{1cm}
{\tilde  \Delta}_{\alpha}(p, n)=\left\{ \left({1\over r}, {1\over s}\right)\in [0, 1]\times [0,1]:\ \
 \min\Big(\frac{1}{r}-\frac{1}{2},\
 \frac{1}{2}-\frac{1}{s}\Big)> -{2\alpha+1\over 2n}, \ \ \alpha- \frac{2\alpha}{p}< \frac{1}{r}-\frac{1}{s}
  \right\}.
 \end{eqnarray*}
\end{remark}

\begin{remark}
Note also that if we know that
$$\left\|S_{R^m}^{-(n+1)/2} (L)\right\|_{r\to s} \leq C
R^{{n}({1\over r}- {1\over s})},
$$
then by interpolation,  we can further extend the range of $r$ and $s$ in point (4) to obtain essentially  the same
optimal results as in the case of the standard Laplace operator.
\end{remark}

\bigskip

\bigskip

\section{Uniform Sobolev inequalities for elliptic operators with constant coefficients
}\label{sec4}
\setcounter{equation}{0}

In this section we will consider  $L^p \to L^q$ uniform boundedness of the
 resolvent of the higher order elliptic differential operators. Let $n\ge 2$ and  $P(\xi)$ be the real
 homogeneous elliptic polynomial of order
 $m \ge 2$ on $\R^ n$ satisfying the following non-degenerate condition:
\begin{eqnarray}\label{eq2.1}
{\rm det}\left( \frac{\partial^2 P(\xi)}{\partial\xi_i\partial\xi_j}\right)_{n\times n}\neq 0,\ \ \  \xi\neq 0,
\end{eqnarray}
which is equivalent to  the fact that  hypersurface
 \begin{eqnarray}\label{eq2.2}\Sigma=\{\ \xi\in \mathbb{R}^n;\ \ |P(\xi)|=1\ \},\end{eqnarray}
has nonzero Gaussian curvature everywhere, see \cite{BNW, St2}.  Without loss of generality,
we may assume  that $P(\xi)>0$ for all $\xi\neq 0$.

Throughout this section we always assume that  $H_0:=P(D)$, where $D=-i(\partial_1, \ldots,
\partial_n)$ and  $P(D)$ is the nonnegative self-adjoint operator associated with
the elliptic polynomial $P(\xi)$  on $L^2(\R ^n)$ and that  $P$ satisfies
conditions \eqref{eq2.1} and \eqref{eq2.2}.  

In our next results we will show  that the operator $\left(H_0-z\right)^{-\alpha}$ can be defined
for all values of $z\in C$ including $z \ge 0$ by  taking limits from upper or lower half-plane gives
different operators if $\alpha >0$, see \eqref{free measure} below.
 Hence it is convenient to introduce  notation $\mathbb{C}^\pm$.
 If $z$ is not a positive real, then this coincides with the standard complex numbers.
 For $z=\lambda>0$ we consider two possibilities $ \lambda +i0$ or $\lambda -i0$.
  The topology of $\mathbb{C}^\pm$ again coincides with topology of $\C$ except of
  set consisting of $ \lambda +i0$ or $\lambda -i0$  where the limit can be only taken from the
  corresponding upper and lower half-planes.

 The following statement is our main  result in this section.


\begin{theorem}\label{thm2.111}  Let $n\ge 2$, $m\ge 2$ and $z\in \mathbb{C}$. Consider
arbitrary auxiliary cutoff function $\psi$ such that
$\psi\in C_0^{\infty}({\mathbb R}),
 \psi(s)\equiv 1$ if $s\in [-2, 2]$ and  $ \psi$ is supported in the interval  $[-4, 4]$.
Assume that  $1/2\le\alpha<(n+1)/2$ for $n\ge3$ and $0<\alpha<3/2$ for $ n=2$.
Suppose  also that
 exponents $(p, q)$ satisfy the following conditions:
\begin{eqnarray}\label{eq2.4-fract}\min\Big(\frac{1}{p}-\frac{1}{2},\
\frac{1}{2}-\frac{1}{q}\Big)> {2\alpha-1\over 2n},
\end{eqnarray}
\begin{eqnarray}\label{eq2.5-fract}
\frac{2\alpha}{n+1}<\Big(\frac{1}{p}-\frac{1}{q}\Big).
\end{eqnarray}
Then there exists  positive constants $C_{p,q}$ independent of $|z|$ such that
\begin{eqnarray}\label{eq2.3-psi}
\|\left(H_0-z\right)^{-\alpha}\psi(H_0/|z|) \|_{p\to q}\le
C_{p,q}\ |z|^{\frac{n}{m}(\frac{1}{p}-\frac{1}{q})-\alpha},\quad
 \forall z \in C^{\pm}\setminus  \{0\}.
\end{eqnarray}
Moreover, for   the same range of $\alpha$  and exponents $(p,q)$  the corresponding
Bochner-Riesz means of order $-\alpha$,   $S^{-\alpha}_\lambda(H_0)=
\frac{1}{\Gamma(1-\alpha )}\left(1-\frac{H_0}{\lambda}\right)^{-\alpha}_+$ is well
 defined  for all $\lambda > 0$ and satisfies similar estimates
\begin{eqnarray}\label{eq2.3-riesz}
\|S^{-\alpha}_\lambda(H_0)  \|_{p\to q}\le C_{p,q}\ \lambda^{\frac{n}{m}(\frac{1}{p}-\frac{1}{q})-\alpha},
\ \ \lambda >  0.
\end{eqnarray}
Next assume  in addition that $m\alpha>n$ or that $\frac{1}{p}-\frac{1}{q}
\le \frac{m\alpha}{n}$, $p \neq 1$ and $q\neq \infty$ for  $m\alpha\le n$.

Then
\begin{eqnarray}\label{eq2.3-fract}
\|\left(H_0-z\right)^{-\alpha}\|_{p\to q}\le C_{p,q}\ |z|^{\frac{n}{m}(\frac{1}{p}-\frac{1}{q})-\alpha}
\end{eqnarray}
for all $z\in \mathbb{C}^{\pm}\setminus\{0\}$.
\end{theorem}

\begin{proof} We only discuss the case  $ m\alpha<n$. The other cases are similar or simpler.
We begin our discussion with verifying  estimates  \eqref{eq2.3-psi}
and \eqref{eq2.3-fract} and postpone considering the operator $S^{\alpha}_\lambda (H_0)$ to
the end of proof. Set $z=re^{i\theta}$ with
$r>0$. If $\delta<|\theta|\le\pi$ for some $\delta>0$,
  then the operator $(H_0-e^{i\theta})^{-\alpha}$ is  a standard constant coefficient pseudo-differential
  operator of order $-\alpha m$ with a symbol $(P(\xi)-e^{i\theta})^{-\alpha}$.
  Hence resolvent estimate  \eqref{eq2.3-fract}
  follows from the standard Sobolev estimates and a scaling argument in $r$.
 A similar argument shows that for any $p \le q$ the multiplier $(H_0-e^{i\theta})^{-\alpha}\psi(H_0)$
 is bounded as as operator from $L^p$ to $L^q$.
Thus we can assume that   $0<|\theta|\le \delta$ and by symmetry it is enough  to consider only the case $\Im z>0$.

We  write $z=(\lambda+i\lambda\varepsilon)^m$ for $\lambda>0$
 and $0<\varepsilon<<1$. Since $|z|\sim\lambda^m$, by homogeneity,
 it suffices to estimate $(H_0-(1+i\varepsilon))^{-\alpha
 }$     and $(H_0-(1+i\varepsilon))^{-\alpha
 }\psi(H_0)$
 for  $0<\varepsilon<<1$.
Let $K^\varepsilon$ be the convolution kernel of $(H_0-(1+i\varepsilon))^{-\alpha }$.  By the inverse  Fourier transform
$$K^\varepsilon=\mathcal{F}^{-1}\Big\{\big(P(\xi)-(1+i\varepsilon)^m\big)^{-\alpha}\Big\}.$$
Note that $K^\varepsilon=K_1+K_2$, where
$$K_1=\mathcal{F}^{-1}\Big( \frac{\psi(P^{1/m}(\xi))}{(P(\xi)-(1+i\varepsilon)^m)^\alpha}\Big)$$
and
$$K_2=\mathcal{F}^{-1}\Big( \frac{1-\psi(P^{1/m}(\xi))}{(P(\xi)-(1+i\varepsilon)^m)^\alpha}\Big).$$
It is clear that to show \eqref{eq2.3-psi}
and \eqref{eq2.3-fract} it is enough to verify  that $K_1$ satisfies \eqref{eq2.3-psi},
whereas  \eqref{eq2.3-fract}  holds for $K_2$.

{\it Estimate \eqref{eq2.3-fract} for  $K_2$.}  To estimate  $K_2$ we note for any $\alpha >0$  it is symbol of order
$-m\alpha$ that is
 $$\Big|D^\beta\Big( \frac{1-\psi(P^{1/m}(\xi))}{(P(\xi)-(1+i\varepsilon)^m)^\alpha}\Big)\Big|\le
 C_\alpha(1+|\xi|)^{-m\alpha-|\beta|}.$$
 Hence $|K_2(x)|\le C_N\ |x|^{m\alpha-n-N}$ for any $N\in \mathbb{N}$ and
 by Young's inequality and interpolation
 \begin{equation}\label{eq2.7-fract}\|K_2*f\|_{q}\le C_{p,q}\|f\|_p
  \end{equation}
  for all $(p,q)$ satisfying $0\le\frac{1}{p}-\frac{1}{q}\le {m\alpha\over n}$ and
  $(p,q)\ne ({n\over m\alpha},\infty), (1,{n\over n-m\alpha})$.

{\it Estimate \eqref{eq2.3-psi} for  $K_1$.} To estimate  $K_1$ we  use the stationary phase principle. We write
  \begin{equation}\label{eq2.8-fract}K_1(x)=\int _{\R ^n}\frac{e^{ix\xi}\
  \widetilde{\psi}(P^{1/m}(\xi))}{(P^{1/m}(\xi)-1-i\varepsilon)^\alpha}\ d\xi=
  \int_0^\infty\frac{s^{n-1}\widetilde{\psi}(s)}{(s-1-i\varepsilon)^\alpha}\
  \Big(\int_\Sigma\frac{e^{isx\omega}d\omega}{|\nabla P(\omega)|}\Big) ds,\end{equation}
where $\widetilde{\psi}(s)=\psi(s)(s^{m-1}+s^{m-2}(1+i\varepsilon)+\ldots+(1+i\varepsilon)^{m-1})^{-\alpha}$.


Note that $K_1$ is the Fourier transform of compactly supported distribution including taking
limits with $\varepsilon $ goes to $\pm0$ so
$|K_1(x)|\le C$  for  all $|x|\le 1$.  To handle the remaining
case $|x|>1$, we recall the following stationary phase formula for the Fourier
transform of a smooth measure
on hypersurface~$\Sigma$
\begin{equation}\label{surface-fract}\int_\Sigma\frac{e^{iy\omega}d\omega}{|\nabla P(\omega)|}
=|y|^{-\frac{n-1}2}c_+(y)e^{i\phi_+(y)}+|y|^{-\frac{n-1}2}c_-(y)e^{-i\phi_-(y)},
\end{equation}
where for say $|y|\ge 1/4$, the coefficients satisfy
\begin{equation}\label{surface estmates-fract}
\Bigl|\frac{\partial^\beta}{\partial y^\beta}c_+(y)\Bigr|
+\Bigl|\frac{\partial^\beta}{\partial y^\beta}c_-(y)\Bigr|
\le C_\alpha|y|^{-|\beta|}, \quad \beta\in \mathbb{N}_0.
\end{equation}
and $\phi_\pm(y)=\langle y,\omega_\pm(y)\rangle$ are smooth homogeneous  function
of degree one. Here  $\omega_\pm(y)$ are the two points of $\Sigma$ such that $\pm\frac{ y}{|y|}$
are the positive normal direction of $\Sigma$ at these points.
Thus by (\ref{eq2.8-fract}) and (\ref{surface-fract})
 \begin{eqnarray}\label{eq2.11-fract}
 K_1(x)&=&\sum_{\pm}\int_0^\infty\frac{s^{n-1}\widetilde{ \psi}(s)}{(s-1-i\varepsilon)^{\alpha}}\
 \Big(|sx|^{-\frac{n-1}2}c_\pm(sx)e^{\pm is\phi_\pm(x)}\Big) ds \nonumber\\
& =& \sum_{\pm}|x|^{-\frac{n-1}{2}}b_\varepsilon^\pm (x) e^{\pm i\phi_\pm(x)},\ \ |x|>1/4,
 \end{eqnarray}
where
$$ b_\varepsilon^\pm(x)=\int_{-\infty}^\infty\frac{(s+1)^{\frac{n-1}{2}}
\widetilde{ \psi}(s+1)}{(s-i\varepsilon)^\alpha} c_\pm((s+1)x)e^{\pm is\phi_\pm(x)} ds.$$
Note that  the  function $s\mapsto (s+1)^{\frac{n-1}{2}}\widetilde{ \psi}(s+1)c_\pm((s+1)x)$ is smooth and compactly supported
 so it is easy to check that
$$|\partial^\beta b_\varepsilon^\pm(x)|\le C_\beta |x|^{\alpha+|\beta|-1}, \ \ |x|>1/4
$$
 uniformly in $\varepsilon>0$.

   Hence in view of (\ref{eq2.11-fract}), we can further smoothly decompose $K_1(x)=K'(x)+K''(x)$
  in  such a way that supp $K'\subset B(0, 1)$ (the unit ball of $\R^n$),  $|K'(x) |\le C$ for all $x$
   and  $K''$  can  be expressed as
   $$K''(x)=\sum_{\pm}|x|^{-\frac{n+1}{2}+\alpha}a_\pm(x) e^{\pm i\phi_\pm(x)},
   $$
where $a_\pm \in C^\infty(\R ^n)$ satisfy $a_{\pm}(x)=0$ for $|x|\le 1/2$ and
$|\partial^\beta a_\pm(x)|\le C_\beta |x|^{-|\beta|}$ for any $\beta\in \mathbb{N}_0$.
By Young's inequality
\begin{equation}\label{eq2.12-fract}
\|K'*f\|_{q}\le C\|f\|_p
\end{equation}
for all $1\le p \le q \le \infty$.

To estimate  $K''$, we note that by the assumption $\alpha<(n+1)/2$  and $|K''(x)|\le (1+|x|)^{-(n+1-2\alpha)/2}$.
Hence
\begin{equation}\label{eq2.13-fract}
\|K''*f\|_{q}\le C\|f\|_p
\end{equation}
for all $(p,q)$ satisfying that $\frac{n-1+2\alpha}{2n}\le\frac{1}{p}-\frac{1}{q}\le1$
but $(p,q)\ne (1, \frac{2n}{n+1-2\alpha}), (\frac{2n}{n-1+2\alpha}, \infty)$. However this argument does not
give the whole range
of pairs $(p,q)$ for which  \eqref{eq2.13-fract} holds. It is possible to  extend it by making use of the
oscillatory factor $e^{\pm i\phi_\pm(x-y)}$ in the integral operator
$$K''*f(x)=\sum_{\pm}\int_{\R^ n}|x-y|^{-\frac{n+1}{2}+\alpha}a_\pm(x-y) e^{\pm i\phi_\pm(x-y)}f(y)dy.$$
In fact, under the assumption that $\Sigma$ has nonzero Gaussian curvature everywhere,
the phase function $\phi_\pm(x-y)$ satisfies the so-called $n\times n$-Carleson-Sj\"olin
conditions, see \cite[p.69]{So} or \cite[p.392]{St2}. Hence the celebrated  Carleson-Sj\"olin
argument  can be used to estimate $K''*f$.

Let $\beta(s)\in C^\infty_c(\R )$ be a such function that supp~$\beta\in [\frac{1}{2},2]$ and
$\sum_{0}^\infty\beta(2^{-\ell}s)=1$ for $s\ge {1/2}$. Set $K''_\ell(x)=
\beta(2^{-\ell}|x|)K''(x)$ for all $\ell=0,1,2,\ldots$ so
 $$
 K''*f(x)=\sum_{\ell=0}^\infty (K_\ell''*f)(x),
 $$
where
$$K_\ell''*f(x):=\int_{\R^ n}|x-y|^{-\frac{n+1}{2}+\alpha}\beta(2^{-\ell}|x-y|)a_\pm((x-y)) e^{\pm i\phi_\pm(x-y)}f(y)dy.
$$
Put $\lambda=2^{\ell}$. By homogeneity
$$(K_\ell''*f)(\lambda x)=\lambda^{\frac{n-1+2\alpha}{2}}\int_{\R ^n} w(x-y)e^{\pm \lambda i\phi_\pm(x-y)}f(\lambda y)dy,$$
where
$
w(x)=|x|^{-\frac{n-1+2\alpha}{2}}\beta(|x|)a_\pm(\lambda x))\in C_c^\infty(\R^n \setminus 0)$
satisfying  $|\partial^\alpha w(x)|\le C_\alpha$ for any $\alpha$.  Now we can apply  Carleson-S\"ojlin argument,
see \cite[p.69]{So} or  \cite[p.392]{St2}, to conclude that
$$ \|K_\ell''*f\|_q\le C \lambda^{-n/p+(n-1+2\alpha)/2}\|f\|_p,\ \ \lambda=2^\ell, \ \ell=0,1,\ldots,$$
\begin{equation}\label{eq2.14-fract}
\|K''*f\|_{q}\le C\|f\|_p,
\end{equation}
where  $q=\frac{n+1}{n-1}p'$, $1\le p<2n/(n-1+2\alpha)$ for all $\alpha\ge 1/2$ if $n\ge3$ and $\alpha>0$ if $n=2$.
 By interpolation between
 \eqref{eq2.13-fract} and \eqref{eq2.14-fract} \begin{equation}\label{eq2.15-fract}
\|K''*f\|_{q}\le C\|f\|_p
\end{equation}
 for all $(p,q)$ such that $\frac{2\alpha}{n+1}< \frac{1}{p}-\frac{1}{q}\le1$ and
 $$\min\Big(\frac{1}{p}-\frac{1}{2},\ \frac{1}{2}-\frac{1}{q}\Big)> {2\alpha-1\over 2n}.$$
 Therefore \eqref{eq2.7-fract}, (\ref{eq2.12-fract}) together with (\ref{eq2.15-fract}) yield  estimate
 (\ref{eq2.3-fract}). 


Next we consider Bochner-Riesz mean operator  $S^{-\alpha}_\lambda(H_0)=
 \frac{1}{\Gamma(1-\alpha )}\left(1-\frac{H_0}{\lambda}\right)^{\alpha}_+$. By homogeneity we can set $\lambda=1$.
Note also  that
\begin{equation*}
(x\pm i0)^\alpha = x_+^\alpha + e^{\pm i\pi \alpha}x_-^\alpha,
\end{equation*}
so
\begin{equation}\label{e3.a}
 e^{ i\pi \alpha}(x- i0)^{-\alpha} - e^{- i\pi \alpha}(x+i0)^{-\alpha}  = 2i\sin(\pi \alpha) x_+^{-\alpha}
 =2i\sin(\pi \alpha)\Gamma(1-\alpha) \chi_+^{-\alpha}=2i\pi\frac{\chi_+^{-\alpha}}{\Gamma(\alpha)}.
\end{equation}
Employing  analytic continuation shows that \eqref{e3.a} is valid for all $\alpha \in \C$,
see also \cite[(3.2.11)]{Ho1}. By \eqref{e3.a}
\begin{eqnarray*}
2i\pi\frac{S^{-\alpha}_1(H_0)}{\Gamma(\alpha)}
= 2i\pi\frac{\chi_+^{-\alpha}(1-H_0)}{\Gamma(\alpha)}
= e^{ i\pi \alpha}\left(H_0-1-i0\right)^{-\alpha} - e^{- i\pi \alpha}
\left(H_0-1+i0\right)^{-\alpha}\\=e^{ i\pi \alpha}\left(H_0-1-i0\right)^{-\alpha}\psi(H_0) - e^{- i\pi \alpha}
\left(H_0-1+i0\right)^{-\alpha}\psi(H_0).
\end{eqnarray*}
Hence we obtain estimate \eqref{eq2.3-riesz} as a direct consequence of \eqref{eq2.3-psi}.
This ends the proof of Theorem~\ref{thm2.111}.
\end{proof}

\medskip

For $\alpha=1$ the formula \eqref{e3.a}
simplifies to following relation
$$
\frac{1}{2\pi i}(\lambda+i0)^{-1}-(\lambda-i0)^{-1}=\delta_0,
$$
see \cite[Example 3.1.13]{Ho1}. Then the  above relation in turn  implies
the well-known absorption principle which connects the spectral projections  $dE_{H_0}(\lambda)$
and the resolvent $ R_0(z)=(H_0-z)^{-1}$
\begin{eqnarray}\label{free measure}
dE_{H_0}(\lambda)f=\frac{1}{2\pi i}(R_0(\lambda+i0)-R_0(\lambda-i0))f.
\end{eqnarray}

We will use the case $\alpha=1$ of Theorem \ref{thm2.111}
and \eqref{free measure} to investigate  the spectral resolution of Schr\"odinger type operators $H=P(D)+V$
with integrable potentials $V$, see Section \ref{sec5} below. Therefore  we summarise
 this particular case of Theorem \ref{thm2.111}
in the following corollary.

\medskip
\begin{coro}\label{boundary operator}
Suppose  that  $n\ge 2$, $m\ge 2$ and the operator $H_0$ satisfies the assumptions of
Theorem~\ref{thm2.111}. Then
\begin{eqnarray}\label{eq2.16B}
\|dE_{H_0}(\lambda)\|_{p\to q}\le C\ \lambda^{\frac{n}{m}(\frac{1}{p}-\frac{1}{q})-1}, \ \ \lambda>0
\end{eqnarray}
for all exponents $(p, q)$ such that
\begin{eqnarray*}
\min\Big(\frac{1}{p}-\frac{1}{2},\
\frac{1}{2}-\frac{1}{q}\Big)> {1\over 2n} \quad \mbox{\rm and} \quad \frac{2}{n+1}<\Big(\frac{1}{p}-\frac{1}{q}\Big).
\end{eqnarray*}
Next assume  in addition that $m>n$ or that $\frac{1}{p}-\frac{1}{q}  \le \frac{m}{n}$,
$p \neq 1$ and $q\neq \infty$ for  $m\le n$.

Then
\begin{eqnarray}\label{eq2.16A}
\|R_0(z)\|_{p\to q}\le C\ |z|^{\frac{n}{m}(\frac{1}{p}-\frac{1}{q})-1}
\end{eqnarray}
for all $z\in \mathbb{C}^\pm\setminus\{0\}$.
\end{coro}

Note that for any natural number $k\in \N$ the operator  $H_0=(-\Delta)^k$
satisfies assumptions of Theorem~\ref{thm2.111} so all the estimates of the statement hold for
poly-harmonic operators. In the case  $\frac{1}{p}-\frac{1}{q}=\frac{2}{n}$ and resolvent
$(z+\Delta)^{-1}$, which corresponds to
the operator    $H_0=-\Delta$, $\alpha=1$ and $n\ge 3$, estimate  \eqref{eq2.16A} from
Corollary~\ref{boundary operator} was obtained by  Kenig, Ruiz and Sogge  in \cite{KRS}.
(In fact, they were able to prove such a uniform
estimate  for a larger class of operators, where the standard Laplace operator $\Delta$
was replaced by a homogeneous  second order constant coefficient
differential operator, non-degenerate but not necessarily elliptic).
In the setting of the Laplace operator on asymptotically conic non-trapping manifolds
estimates \eqref{eq2.16A} were obtained by Guillarmou and  Hassell in \cite{GH}.
For other results of this type see also \cite{Gu} and the references within.

The classical Bochner-Riesz means operators
  $S^{-\alpha}_1(-\Delta) $ with a negative index $-\alpha$ corresponding to the standard Laplace operator  have
   been studied by many authors, see for example \cite{Bak, Bor, CS, Gu} and references therein.

\bigskip

%
%
%
%
%

\section{Restriction type estimates  for Schr\"odinger operators $P(D)+V$ }\label{sec5}
\setcounter{equation}{0}

In this section we will establish $L^p \to L^q$ estimates for the  perturbed resolvent
$R_H(z)=(z-H)^{-1}$ for any ${z\neq0}$, where  $H:=H_0+V=P(D)+V$ is a self-adjoint operator
with the real valued potential $V$.
For simplicity we assume that  $V\ge0$ belong to $L^1_{\rm loc}(\R^n)$.
Then it is well-known that the operator $H$ can be defined as a self-adjoint extension by
the following non-negative closed  form
\begin{eqnarray}
 \label{eq3-01}
 Q_V(f):=\int_{\mathbb{R}^n} P(\xi)\ |\widehat{f}(\xi)|^2 d\xi+\int_{\mathbb{R}^n} V|f|^2 dx
 \end{eqnarray}
for all $f\in W^{m,2}(\mathbb{R}^{n})$ such that  $\int V|f|^2dx<\infty$.

 In order to obtain the estimates for the  resolvent $R_H(z)=(z-H)^{-1}$, a crucial step will
be to pass from \eqref{eq2.16A} to a similar estimate for $H$ by writing the
standard perturbation formula:
 \begin{eqnarray}\label{eq3-01a}
 R_H(z)f=R_0(z)(I+VR_0(z))^{-1}f,\ \  {\rm Im}z\neq0.
 \end{eqnarray}
 In the next step we will study the boundary
  resolvent $R_H(\lambda\pm i0)$ and by Stone's  formula
\begin{eqnarray}\label{eq3-01b}
dE_{H}(\lambda)f=\frac{1}{2\pi i}(R_H(\lambda+i0)-R_H(\lambda-i0))f
\end{eqnarray}
deduce  restriction type estimates for the spectral projection measure $dE_{H}(\lambda)$. Using the notation
$\C^{\pm}$ introduced at the beginning of section \ref{sec4} for   $\lambda>0$ and
 $z=\lambda\in \mathbb{C}^\pm$ we always assume that  $R_H(z)=R_H(\lambda\pm i0)$.
We first verify
the following lemma.

 \begin{lemma}\label{pro3.1}
 Suppose   $n\ge 2$, $m\ge 2$  and  that  $H_0$ satisfies assumptions  of Theorem~\ref{thm2.111}.
Assume also that exponents  $(p,q)$ satisfy all conditions listed in
 Corolary \ref{boundary operator}, $\frac{1}{r}=\frac{1}{p}-\frac{1}{q}$  and that $0\le V\in L^r(\R^n) $. Then
 \begin{eqnarray}
 \label{eq3-02}
 \|VR_0(z)\|_{p\to p}\le C\|V\|_r\ |z|^{\frac{n}{m}(\frac{1}{p}-\frac{1}{q})-1},
 \ \  \forall z\in \mathbb{C}^\pm\setminus\{0\}.
  \end{eqnarray}
  \end{lemma}

  \begin{proof}
  Let $M_V$ be the multiplication operator defined by $M_Vf=V(x)f(x)$. Then by
   H\"older's inequality $\|M_V\|_{q \to p}\le \|V\|_r$ and  by
  Corollary \ref{boundary operator}
  $$\|VR_0(z)\|_{p\to p}\le \|M_V\|_{q\to p} \|R_0(z)\|_{p\to q}\le
  C\|V\|_r \ |z|^{\frac{n}{m}(\frac{1}{p}-\frac{1}{q})-1},
   \ \ z\in \mathbb{C}^\pm\setminus\{0\}.$$
  \end{proof}

Assume now that the exponent $p$
satisfies the relation $ \max\Big(\frac{2n}{n+m},1\Big)<p <\frac{2(n+1)}{n+3}$.
Note that then the pair $(p,p')$ satisfies  all conditions from Corollary \ref{boundary operator}.
This yields  the following corollary.

   \begin{coro}\label{coro 3.2}
   Suppose again that $n\ge 2$, $m\ge 2$, $H_0$ satisfies assumptions  of
   Theorem~\ref{thm2.111} and that  $0\le V\in L^{\frac{n+1}{2}}(\R^n)\cap L^{s}(\R^n) $
   where  $s=\max\Big(\frac{n}{m},1\Big)$.
Then there exists a constant $C>0$ 
such that 
  \begin{eqnarray}
 \label{eq3-02a}
 \|VR_0(z)\|_{p\to p}\le C\ |z|^{\frac{n}{m}(\frac{1}{p}-\frac{1}{p'})-1}, \ \ \forall z\in \mathbb{C}^\pm\setminus\{0\}
  \end{eqnarray}
 for all
$\max\Big( \frac{2n}{n+m}, 1\Big)<p <\frac{2(n+1)}{n+3}.$

 In particular, there exists a constant $\delta>0$ such that
 the operator $I+VR_0(z)$ is invertible on $L^p(\mathbb{R}^{n})$ and
 \begin{eqnarray}
 \label{eq3-02b}
 \sup_{|z|>\delta}\|(I+VR_0(z))^{-1}\|_{p\to p}\le C.
  \end{eqnarray}
  for all  $z\in \mathbb{C}^\pm\cap\{\ |z|\ge \delta\ \}$.
\end{coro}

\begin{proof}
We only discuss the case  $m>n$ because the proof for the case $m \le n$ is similar.
In the considered situation  $\frac{2n}{n+m}<p
<\frac{2(n+1)}{n+3}$ so if we set $\frac{1}{r}=\frac{1}{p}-\frac{1}{p'}$, then ${n\over m} <r< \frac{n+1}{2}$.
Hence  by Lemma~\ref{pro3.1}
$$\|VR_0(z)\|_{p\to p}\le C \|V\|_r \ |z|^{\frac{n}{m}(\frac{1}{p}-\frac{1}{p'})-1}\le C
\|V\|^\theta_{(n+1)/2}\|V\|^{1-\theta}_{n/m}\ |z|^{\frac{n}{m}(\frac{1}{p}-\frac{1}{p'})-1} \quad
\forall z\in \mathbb{C}^\pm\setminus\{0\},$$
where  $\theta =(\frac{1}{p}-\frac{1}{p'}-\frac{m}{n})/(\frac{2}{n+1}-\frac{m}{n}) $.  Note that
$$\|V\|^\theta_{(n+1)/2}\|V\|^{1-\theta}_{n/m}\le (1+\|V\|_{(n+1)/2})(1+\|V\|_{n/m}),$$
hence there exists a constant $C$ depending on $n, m , V$  such that estimate (\ref{eq3-02a}) holds.

Next we verify estimate (\ref{eq3-02b}). Note that $\frac{n}{m}(\frac{1}{p}-\frac{1}{p'})-1<0$ so there exists a
 constant $\delta>0$ such that  $\|VR_0(z)\|_{p\to p}\le \frac{1}{2}$
 for all $|z|>\delta$.
 By the standard Neumann series argument the last estimate  yields  $\|(1+VR_0(z))^{-1}\|_{p\to p}\le 2.$
  \end{proof}


\vskip0.3cm
In order to use Corollary \ref{coro 3.2} to establish the $L^p$-estimates of the spectral
projections measure $dE_H(\lambda)$, we need the following lemma essentially due to H\"ormander \cite[Chapter 14]{Ho1}. 
\begin{lemma}\label{le5.3} Let $0\le V\in L^\infty(\R^n)$  with compact support. Then
the equality
  \begin{eqnarray}
 \label{eq5-01}
 \langle R_H(z)f, g\rangle=\langle R_0(z)(I+VR_0(z))^{-1}f, g\rangle, \ \ f,\ g\in {\mathscr S} (\R^n)
  \end{eqnarray}
holds for all $z\in \mathbb{C}^\pm \setminus \big(\{0\}\cup \Lambda \big)$, where $\Lambda$
is the set of positive discrete eigenvalues of $H=P(D)+V$.  In particular, the functions on
the both sides of (\ref{eq5-01}) are continuous on
$ z \in \mathbb{C}^\pm \setminus \big(\{0\}\cup \Lambda \big)$ and analytic in its interior.
\end{lemma}

\begin{proof} The potential $V$ is a bounded and compactly supported function so
it is the short range perturbation  of $P(D)$, see  H\"ormander \cite[page 246 of Chapter 14]{Ho1}.
Therefore, the equality \eqref{eq5-01} immediately follows from H\"ormander \cite[Theorem 14.5.4 of Section 14.5]{Ho1}.
\end{proof}
\vskip0.3cm
\begin{remark}
The set $\Lambda$ is the point spectrum of $H$ with the finite multiplicity,
 which is discretely embed into positive real line. It would be interesting
 to show that $\Lambda$ is empty for general higher order elliptic operator $P(D)+V$.
 In the case of second order operators, the absence of positive eigenvalues has been
 studied in depth  by many authors and confirmed for  potential with  decay   of the
 order $o(1/|x|)$ and some integrable class, see e.g.  H\"ormander
\cite[Chapter 14]{Ho1}, Koch and Tataru \cite{KT} and references therein.
\end{remark}
\vskip0.2cm

\begin{proposition}\label{prop 3.3}
Under the assumptions of  Corollary \ref{coro 3.2} there exists a constant $\delta>0$ such that
 \begin{eqnarray}
 \label{eq3-03}
 \|R_H(z)\|_{p\to p'}\le C\ |z|^{\frac{n}{m}(\frac{1}{p}-\frac{1}{p'})-1}, \ \ z\in \mathbb{C}^\pm \cap \{\ |z|\ge \delta\ \}
  \end{eqnarray}
 and
  \begin{eqnarray}
 \label{eq3-03a}
 \|dE_H(\lambda)\|_{p\to p'}\le C\ \lambda^{\frac{n}{m}(\frac{1}{p}-\frac{1}{p'})-1}, \ \ \lambda\ge \delta
  \end{eqnarray}
 for all  $\max\Big( \frac{2n}{n+m}, 1\Big)<p <\frac{2(n+1)}{n+3}.$
\end{proposition}
\begin{proof}
It suffices to prove  estimate (\ref{eq3-03})
for large $|z|>\delta$ since  estimates (\ref{eq3-03a}) immediately follow from
(\ref{eq3-03}) by Stone's formula (\ref{eq3-01b}).  Similarly, we only discuss the case $m>n$.
Given  $\frac{2n}{n+m}<p<\frac{2(n+1)}{n+3}$, then $V\in L^r(\R^n)$ with $r={p\over 2-p}
\in ({n\over m},\frac{n+1}{2})$. In order to obtain (\ref{eq3-03}),  we need to  establish
 equality (\ref{eq5-01}) for $|z|>\delta$ as $0\le V\in L^r(\R^n)$.   Firstly, we can take a
  monotonically increasing sequence of $0\le V_k\in L^\infty$ with compact support such that
   $V_k(x)$ converges to $V(x)$  as $k\rightarrow\infty$ in both pointwise and  $L^r$ norm sense.
     By Lemma \ref{le5.3}  for every  $k$
 \begin{eqnarray}
 \label{eq5-02}
\langle R_{H_k}(z)f, g\rangle=\langle R_0(z)(I+V_kR_0(z))^{-1}f, g\rangle, \ \  f,\ g\in {\mathscr S} (\R^n)
  \end{eqnarray}
 for all $z\in \mathbb{C}^\pm \setminus \big(\{0\}\cup \Lambda_k \big)$, where $\Lambda_k$
 is the point spectrum  of  the operator $H_k=P(D)+V_k$.  Note that $\|V_k\|_{r}\le \|V\|_{r} $
 for every  $k$ so by Corollary \ref{coro 3.2} the estimates
 $\|(I+V_kR_0(z))^{-1}\|_{p-p}\le C$ hold  uniformly for all $|z|>\delta$, which implies that
 the set $|z|> \delta$ has no eigenvalues of $H_k$ (i.e. $\Lambda_k\cap \ \{\ |z|>\delta\ \}$ is
  an empty set).  Indeed, if $\lambda\in \Lambda_k$ and $\lambda>\delta$, then
there is a $0\neq g_k\in L^2$ such that
$$
(\lambda-H_k)g_k=(\lambda-P(D)-V_k)g_k=0.
$$
Set $f_k=V_kg_k$. Then $g_k=R_0(\lambda+i0)f_k$ and $(I+V_kR_0(\lambda +i0))f_k=0$. Note that
 $ f_k\neq 0$  and it belongs to $L^p(\R^n)$ by H\"older's inequality. This contradicts the
 existence of inverse   $(I+V_kR_0(\lambda\pm i0))^{-1}$ as bounded operator on $L^p(\R^n)$.
   Thus  equality (\ref{eq5-02}) actually holds  on $\mathbb{C}^\pm \cap \{\ |z|\ge \delta\ \}$,
    the both sides of which  are analytic on the interior of $\mathbb{C}^\pm \cap \{\ |z|\ge \delta\ \}$
	and continuous up to their boundary.

Now we can extend the argument above to  a  general potential $0\le V\in L^r$ by taking limit in
 equality (\ref{eq5-02}). Note that for $|z|\ge \delta$,
$$(I+V_kR_0(z))^{-1}-(I+VR_0(z))^{-1}=(I+V_kR_0(z))^{-1}\Big((V_k-V)R_0(z)\Big)(I+VR_0(z))^{-1}.$$
By Corollary \ref{coro 3.2}
$$\|(I+V_kR_0(z))^{-1}-(I+VR_0(z))^{-1}\|_{p-p}\le C \|V_k-V\|_r\rightarrow 0$$
as $k\rightarrow\infty$.
Hence the left side of (\ref{eq5-02}) converges uniformly to
 $\big \langle R_0(z)(I+VR_0(z))^{-1}f, g\big \rangle$ on any compact subset
  $K $ of $\mathbb{C}^\pm \cap \{\ |z|\ge \delta\ \}$  and  the function
  $z \to \big \langle R_0(z)(I+VR_0(z))^{-1}f, g\big \rangle$  is continuous
  on the set $  \mathbb{C}^\pm\cap \{\ |z|\ge \delta\ \} $ and analytic in its
  interior. On the other hand, for any $k$ we define the closed forms $Q_{V_k}$ associated with $H_k=P(D)+V_k$ by
  \begin{eqnarray}
 Q_{V_k}(f):=\int_{\mathbb{R}^n} P(\xi)\ |\widehat{f}(\xi)|^2 d\xi+\int_{\mathbb{R}^n} V_k|f|^2 dx, \ \  f\in W^{m,2}(\mathbb{R}^{n}).
 \end{eqnarray}
Then since the increasing sequence of nonnegative
  forms $Q_{V_k}(f)$  monotonically converges to  the form $Q_V(f)$ defined in (\ref{eq3-01}),
  so  by Kato \cite[Theorem 3.13a]{Ka} it follows that  $\langle R_{H_k}(z)f, g\rangle $
  converges to $\langle R_{H}(z)f, g\rangle $ for each $\Re z<0$. Hence on the common
  domain  $\mathbb{C}^\pm \cap \{\ \Re z<-\delta\ \}$
\begin{eqnarray}
 \label{eq5-03}
\langle R_H(z)f, g\rangle=\langle R_0(z)(I+VR_0(z))^{-1}f, g\rangle, \ \  f,\ g\in {\mathscr S} (\R^n).
  \end{eqnarray}
Note that  both side of (\ref{eq5-03})  extend analytically into the interior of
$\mathbb{C}^\pm\cap \{\ |z|\ge \delta\ \}$, so by the uniqueness of analytic function
 extension  equality (\ref{eq5-03}) holds on  $\mathbb{C}^\pm\cap \{\ |z|\ge \delta\ \}$.
  Thus   estimate \eqref{eq3-03} follows from  estimates \eqref{eq2.16A} and \eqref{eq3-02b}.
\end{proof}

If  $m<n$ and the $L^{n\over m}$ norm of  potential $V$ is  small then we can  extend
Proposition  \ref {prop 3.3} to a  small  values of frequency $\lambda$.
The proof is based on the  uniform Sobolev $L^p \to L^q$ estimates
of the free resolvent  $R_0(z)$, which yields the required estimates for pairs $(p,q)$
on the Sobolev line $1/p-1/q=m/n$.

\begin{proposition}\label{prop 3.4}
 Suppose  that $n\ge 2$, $m\ge 2$, $H_0$ satisfies assumptions  of Theorem~\ref{thm2.111}
 and that  $0\le V\in L^{n\over m}(\mathbb{R}^{n})$. There exists a constant
$c_0>0$ such that when $\|V\|_{n\over m}\le c_0$,  then
\begin{eqnarray}
 \label{eq3-03b}
 \|R_H(z)\|_{p\to p'}\le C\ |z|^{\frac{n}{m}(\frac{1}{p}-\frac{1}{p'})-1}, \ \quad \forall z\in \mathbb{C}^\pm\setminus\{0\}
  \end{eqnarray}
 and
  \begin{eqnarray}
 \label{eq3-03c}
 \|dE_H(\lambda)\|_{p\to p'}\le C\ \lambda^{\frac{n}{m}(\frac{1}{p}-\frac{1}{p'})-1}, \ \ \lambda>0
  \end{eqnarray}
 for all
\begin{eqnarray}
 \label{eq3-03d}
 \frac{2n}{n+m}\le p <\min \Bigg(\frac{2(n+1)}{n+3}, {n\over m}\Bigg).
 \end{eqnarray}
\end{proposition}

\begin{proof}
If $p$ satisfies  condition (\ref{eq3-03d}), then there exists $q>1$ such that
the pair $(p, q)$ lies on the Sobolev line $\frac{1}{p}-\frac{1}{q}=\frac{m}{n}$ and
$\frac{1}{p}-\frac{1}{2}>\frac{1}{2n}$. This means that  all  conditions for exponents $(p,q)$
listed in Corollary \ref{boundary operator}
hold. Hence by Lemma \ref{pro3.1}
$$
\|VR_0(z)\|_{p\to p}\le C \|V\|_{\frac{n}{m}}, \ \ \forall z\in \mathbb{C}^\pm\setminus\{0\}.
$$
Setting $c_0=C$ in the above estimate ensures that $\|VR_0(z)\|_{p\to p}\le \frac{1}{2}$
and
$$\sup_{|z|>0}\|(I+VR_0(z))^{-1}\|_{p\to p}\le 2.
$$
Now  the estimates (\ref{eq3-03b}) and (\ref{eq3-03c}) follows  from  (\ref{eq3-01a}) and Stone's formula (\ref{eq3-01b}).
\end{proof}

%

 In Theorem \ref{th5.6} below we shall  extend estimates
(\ref{eq3-03d}) to the large range  $1\le p<\min \Big(\frac{2(n+1)}{n+3},{n\over m}\Big)$.
In particular, if
$m=2, n\ge3$, this  range is optimal and coincides with one described in  Corollary \ref{boundary operator}.
The argument we use is based on Lemmas \ref{lemma 3.6} and \ref{lemma 3.7} below.

\begin{lemma}\label{lemma 3.7}
Assume that $n>m\ge 2$, $H=H_0+V$ for a potential   $0\le V\in L_{\rm loc}^{1}(\mathbb{R}^{n})$
and that  $H_0=P(D)$ satisfies the assumptions  of Theorem~\ref{thm2.111}.
There exists  a constant $c_0>0$ such that if
$$
\sup_{y\in \mathbb{R}^{n}}\int_{\mathbb{R}^{n}}\frac{V(x)}{|x-y|^{n-m}}dx\le c_0,
$$
then   estimate
\begin{eqnarray}
\label{eq3-04}
\|(I+tH)^{-1}\|_{p\to q}\le C_{p,q} t^{-\frac{n}{m}(\frac{1}{p}-\frac{1}{q})},  \ \  t>0
\end{eqnarray}
holds for all pairs $(p,q)$ such that 
$0\le 1/p-1/q < {m/n}$.
Moreover, for  $k\in \mathbb{N} $ large enough
\begin{eqnarray}
\label{eq3-04a}
\|(I+tH)^{-k}\|_{p\to q}\le C_{k,p,q} t^{-\frac{n}{m}(\frac{1}{p}-\frac{1}{q})},  \ \  t>0
\end{eqnarray}
hold for all $1\le p\le q\le \infty$.
\end{lemma}

\begin{proof}
We first prove that
\begin{equation}\label{step1}
\|tV(I+tP(D))^{-1}\|_{1\to 1}\le \frac{1}{2}.
\end{equation}
Note that
\begin{eqnarray}\|tV(I+tP(D))^{-1}\|_{1\to 1}
&\le& \|VP(D)^{-1}\|_{1\to 1}\|tP(D)(I+tP(D))^{-1}\|_{1\to 1}\nonumber\\
&\le& C \|VP(D)^{-1}\|_{1\to 1}.\nonumber
\end{eqnarray}
Since  the fundamental solution of $P(D)$ is bounded by $O(|x|^{m-n})$
so
$$\|VP(D)^{-1}\|_{1\to 1}\le C \sup_{y\in \mathbb{R}^{n}}\int_{\mathbb{R}^{n}}\frac{|V(x)|}{|x-y|^{n-m}}dx.
$$
Thus  when $c_0$  is enough small then \eqref{step1} holds.


Next we prove  estimates \eqref{eq3-04}.
Taking  adjoint and using interpolation we reduce the proof  to the case $p=1$. By  \eqref{step1} and Neumann series argument
 $$
 \Big\|\Big(I+tV(I+tP(D)^{-1})\Big)^{-1}\Big\|_{1\to 1}\le 2.
 $$
Writing  the standard  perturbation formula in our notation yields
$$
(I+tH)^{-1}=(I+tP(D))^{-1}\Big(I+tV(I+tP(D)^{-1})\Big)^{-1}
$$
for  all $t>0$. Hence if  $0\le 1-1/q <{m/n}$, then it follows from the Sobolev embedding that
$$\|(I+tH)^{-1}\|_{1\to q}\le 2\ \|I+tP(D)^{-1}\|_{1\to q}\le Ct^{-\frac{n}{m}({1}-\frac{1}{q})}.$$

To verify  estimate (\ref{eq3-04a}) we  note that  by \eqref{eq3-04} for any $k\in \mathbb{N}$,
   $$
   \|(I+tH)^{-k}\|_{p\to p}\le C
   $$ for all $1\le p\le \infty$.
 It follows from an interpolation argument that it suffices to show that for  $k\in \mathbb{N}$ large enough
 $$
 \|(I+tH)^{-k}\|_{1\to \infty }\le C t^{-\frac{n}{m}},  \ \  t>0.
 $$
To prove the above relation we iterate k-times the resolvent $ (I+tH)^{-1}$.   Choose
  $k\in \mathbb{N}$ such that  $0<\frac{1}{k}\le m/n $.
Let $p_1=1$ and for  each $1\le i\le k$ we define $p_{i+1}$ by putting  $1/p_i-1/p_{i+1}=\frac{1}{k}$. Note that
$p_{k+1}=\infty$.
   By  estimate \eqref{eq3-04}
$$
\|(I+tH)^{-k_0}\|_{1\to \infty}\le \prod_{i=1}^{k_0}
\|(I+tH)^{-1}\|_{p_i\to p_{i+1}}\le Ct^{-\frac{n}{m}},  \ \  t>0.
$$
This  concludes the   proof of Lemma \ref{lemma 3.7}.
\end{proof}

The following result is a consequence of Lemma \ref{lemma 3.7}.

\begin{theorem}\label{th5.6}
Assume that $n>m\ge 2$, $H=H_0+V$ for a potential   $0\le V\in L_{\rm loc}^{1}(\mathbb{R}^{n})$
and that  $H_0=P(D)$ satisfies the assumptions  of Theorem~\ref{thm2.111}. There exists a
 small constant $c_0>0$ such that if
\begin{eqnarray}
 \label{poten-norm}\|V\|_{n\over m}+\sup_{y\in \mathbb{R}^{n}}\int_{\mathbb{R}^{n}}\frac{V(x)}
 {|x-y|^{n-m}}dx\le c_0, \end{eqnarray}
 then the estimate
 \begin{eqnarray}
 \label{eq3-06}
 \|dE_H(\lambda)\|_{p\to p'}\le C\ \lambda^{\frac{n}{m}(\frac{1}{p}-\frac{1}{p'})-1}, \ \ \lambda>0
  \end{eqnarray}
holds for all
 $1\le p <\min \Big(\frac{2(n+1)}{n+3},\frac{n}{m}\Big)$.

\end{theorem}

\begin{proof}
This is a consequence  of Proposition \ref{prop 3.4}, Lemma \ref{lemma 3.6} and Lemma \ref{lemma 3.7}.
\end{proof}

\begin{remark}\label{remark5.9} If $m=2$ and $n\ge 3$ (for example $H=-\Delta +V$),
 then  $\frac{2(n+1)}{n+3}\le {n\over 2}$, hence we  obtain almost the optimal range
 $1\le p <\frac{2(n+1)}{n+3}$ for  estimate
\eqref{eq3-06}. The endpoint estimate, $p=\frac{2(n+1)}{n+3}$,  can also
be verified  by the following  resolvent estimate:
 $$\Big\|(Q(D)-z)^{-1}\Big\|_{\frac{2(n+1)}{n-1}\to \frac{2(n+1)}{n+3}}\le C\ |z|^{-\frac{1}{n+1}},
 \ \quad \forall z\in \mathbb{C}^\pm\setminus\{0\}, $$
where $Q(D)$ is any second order homogeneous elliptic operator, see e.g. Stein \cite[page 370]{St1}.
\end{remark}

\section{Applications} \label{sec6}
\setcounter{equation}{0}

 As an illustration of our results we will  discuss a class  of possible applications, which
  include
 $m$-th order elliptic  operators with some positive  potentials and  Schr\"odinger operator with the
 inverse-square potential.

  \subsection{Higher order elliptic operators with   potentials}
 In this subsection, we show  H\"ormander-type spectral multiplier theorem  for
elliptic  $m$-th order  operators perturbed by  potentials.  Our discussion requires  the following lemma.
\begin{lemma}\label{le6.1}
Let $P(D)$ be a positive elliptic  $m$-th order  homogeneous operator  and
$0\le V\in L^1_{\rm loc}(\mathbb{R}^{n})$ be a potential.
  Then the semigroup $e^{-tH}$ generated
 by $H=P(D)+V$ satisfies the $m$-th order Davies-Gaffney estimates, that is,
 there exist constants $c, C>0$ such that for all $t>0$ and all $x,\ y\in \mathbb{R}^{n}$,
\begin{eqnarray}
 \label{e6.1}
 \|P_{B(x, t^{1/m})}e^{-tH}P_{B(y,t^{1/m})}\|_{2\to 2}\le C
 \exp \left(-c\left(\frac{|x-y|}{t^{1/m}}\right)^{{m}/{m-1}}\right).
 \end{eqnarray}
 \end{lemma}

\begin{proof} The proof of \eqref{e6.1} is based on
   the ideas of Barbatis, Davies \cite{B-D} and Dungey \cite{Du}.
Consider the set of linear functions $\psi \colon \R^n \to \R$ of the form
$\psi(x)=a \cdot x$, where $a=(a_1, \ldots, a_n) \in {\bf S}^{n-1}$. Then
for $\lambda \in \R$ we consider the conjugated operator
$$
H_{\lambda \psi}=e^{-\lambda \psi}He^{\lambda \psi}
=P_{\lambda\psi}(D) +V,
$$
where $P_{\lambda \psi}(D)= e^{-\lambda \psi}P(D) e^{\lambda \psi}=P(D-i\lambda a) $.
Note that $V\ge 0$, then there exists some constant $d_0>0$ such that
$$\Re\langle H_{\lambda\psi}f, f\rangle\ge \Re \langle P_{\lambda\psi}(D)f, f\rangle \ge - d_0\lambda^m\|f\|^2_2.$$
Let $f_t=e^{-tH_{\lambda\psi}}f$ for $f\in L^2(\mathbb{R}^{n})$. Then
\begin{eqnarray}
\frac{d}{dt}\|f_{t}\|^{2}_{2}&=&-\langle
H_{\lambda\psi}f_{t},f_{t}\rangle-\langle f_{t},H_{\lambda\psi}f_{t}\rangle\nonumber\\
&=&-2\Re\langle H_{\lambda\psi}f_t, f_t\rangle\le 2d_0\lambda^{m}\|f_{t}\|^{2}_{2},\nonumber
\end{eqnarray}
which implies that
\begin{eqnarray}\label{e6.3}
\|e^{-tH_{\lambda\psi}}f\|_{2}\leq e^{2d_0\lambda^{m}t}\|f\|_{2}.
\end{eqnarray}
Note that $e^{-tH_{\lambda\psi}}=e^{-\lambda \psi}e^{-tH}e^{\lambda\psi}$. We get that
$$
\|e^{-\lambda \psi}\exp(-t H)e^{\lambda \psi}\|_{2\to 2} \le e^{c\lambda^mt}.
$$
Now we consider  $a=(a_1,a_2,a_3) \in {\bf S}^{n-1}$ such that $\psi(x)-\psi(y)=|x-y|$.
Then
$$
\big\|P_{B(x, t^{1/m})} e^{-tH} P_{B(y, t^{1/m})}\big\|_{2\to {2}}
\le C e^{c\lambda^mt-\lambda(|x-y|-2t^{1/m})}.
$$
Taking infimum over $\lambda$ in the above inequality, we obtain estimate \eqref{e6.1}.
\end{proof}

\begin{remark}\label{GEm}
Let $H=P(D)+V$ with  $0\le V\in L_{\rm loc}^{1}(\mathbb{R}^{n})$. If $m>n$ or $m=2$, then it is
well-known that the semigroup $e^{-tH}$ satisfies the Gaussian  estimates \eqref{GE}
 \begin{eqnarray}\label{e6.2'}
 |p_t(x,y)|\leq C t^{-{n\over m}} \exp\Big(-c\Big({|x-y|^m \over    t}\Big)^{1\over m-1}\Big)
\end{eqnarray}
for some $C, c>0$.    On the other hand, if $4\le m\le n$, then generally, the Gaussian bound
of $e^{-tH}$ may  fail to hold. For these results and further details, see
\cite{D2}, \cite{DDY} and therein references.
\end{remark}


We are now  able to state some results describing spectral multipliers for $m$-th order elliptic
operators with  positive potentials $V$ on ${\mathbb R^n}.$  As above, let $H=P(D)+V$  and
 $0\le V\in L_{\rm loc}^{1}(\mathbb{R}^{n})$. If $n<m$, then by Remark \ref{GEm}, the Gaussian
 estimate (\ref{e6.2'}) holds , which immediately implies   Davies-Gaffney estimate \eqref{e6.1}
 and   condition \eqref{DG}, see Section \ref{sec2}. Hence it follows from point (i)
 of Proposition \ref{prop2.2} that for any  $1\le p < 2$, the spectral multiplier operator $F(H)$ is
 bounded on $L^q(\R^n)$ for all $p<q<p'$ if a   bounded Borel function $F: [0, \infty) \to \CC$ satisfies
$\sup_{t>0}\|\eta\, \delta_tF\|_{C^k}<\infty $ for some  $k>n(1/p-1/2)$ and  some  non-zero auxiliary function
$\eta \in C_c^\infty(0,\infty)$. In particular,  as $p=1$, this exactly corresponds to a spectral
multiplier version of the classical  Mikhlin theorem. However,
for the cases $n>m$, we need to impose a non-degenerate condition \eqref{eq2.1}
on $P(\xi)$.
Now based on  estimate \eqref{eq3-06} and  Davies-Gaffney estimate \eqref{e6.1},
  the following   H\"ormander type spectral multipliers result for $H=P(D)+V$ holds.

\begin{theorem}\label{th6.2}
 Suppose  that  $n>m\ge 2$, $H_0=P(D)$ satisfies the  assumptions  of Theorem~\ref{thm2.111} and
  that  $0\le V\in L^{n\over m}(\mathbb{R}^{n})$.  There exists a
 small constant $c_0>0$ such that if
\begin{eqnarray}
 \label{poten-norm2}\|V\|_{n\over m}+\sup_{y\in \mathbb{R}^{n}}\int_{\mathbb{R}^{n}}
 \frac{V(x)}{|x-y|^{n-m}}dx\le c_0, \end{eqnarray}
then for any $1\leq p< \min\left(\frac{2(n+1)}{n+3}, {n\over m}\right)$ and  any   bounded Borel
function $F: [0, \infty) \to \CC$ satisfying
$\sup_{t>0}\|\eta\, \delta_tF\|_{W^{\alpha, 2}}<\infty $ for  $\alpha>n(1/p-1/2)$,
 the operator
$F(H)$ is bounded on $L^q(\R^n)$ for all $p<q<p'$.
In addition,
\begin{eqnarray*}
   \|F(H)  \|_{q\to q}\leq    C_\alpha\sup_{t>0}\|\eta\, \delta_tF\|_{W^{\alpha, 2} }.
\end{eqnarray*}
\end{theorem}

\begin{proof} Note that  $n(1/p-1/2)>1/2$ for all $p$ above, hence  by Theorem~\ref{th5.6} and Lemma~\ref{le6.1},
Theorem~\ref{th6.2} follows from point (ii) of Proposition ~\ref{prop2.2}.
See also  \cite[Theorem 5.1]{SYY}.
\end{proof}

Note that for any $1\le p\le 2$, the function $(1-\lambda)_+^\delta\in W^{n(1/p-1/2), 2}$
 if $\alpha>n(1/p-1/2)-1/2$. Hence as a corollary, we can apply Theorem \ref{th6.2} to discuss the
  bounds of { Bochner-Riesz means }
 $S^{\alpha}_R\left({H}\right) \ \ $
where $H=P(D)+V$.

\begin{coro}\label{coro6.4}
Let $n, m, P(D)$ and $ V$ satisfy the same conditions as Theorem \ref{th6.2}. Then it follows
that for any $1\leq p< \min\Big(\frac{2(n+1)}{n+3}, {n\over m}\Big)$
 and $\alpha>n(1/p-1/2)-1/2$, Bochner-Riesz means
\begin{eqnarray*}
\sup_{R>0}\left\|S^{\alpha}_R\left({H}\right)\right\|_{r\to r}\leq C
\end{eqnarray*}
unifomly hold for any $p<r<p'$. In particular, we can take $r=p$
and $1\le p\le \frac{2(n+1)}{n+3}$ if $m=2$ and $n\ge 3$.
\end{coro}

\bigskip

 \subsection{Schr\"odinger operator  with the inverse-square potential}
We consider the spectral estimates for Schr\"odinger operator $H=-\Delta+V$ with  the inverse
square potential, that is $V(x) = c/| x |^2$.   Fix $n > 2$ and assume that
$  -{(n-2)^2/4}< c $.   Note that the potential V(x) does not satisfy with
   condition  (\ref{poten-norm2}) even if $c$ is very small. Hence in the subsection
 we will study this potentials case.  First, define by quadratic form method
$ H= -\Delta + V$ on $L^2({\mathbb R^n}, dx)$.
The classical Hardy  inequality
\begin{equation}\label{hardy1}
- \Delta\geq  {(n-2)^2\over 4}|x|^{-2},
\end{equation}
shows that  for all $c > -{(n-2)^2/4}$,  the self-adjoint operator $H$  is   non-negative.
Set  $p_c^{\ast}=n/\sigma$, $\sigma= \max\{ (n-2)/2-\sqrt{(n-2)^2/4+c}, 0\}$. If $c \ge 0$, then the semigroup $\exp(-tH)$
is pointwise bounded by the Gaussian semigroup and hence acts on all $L^p$ spaces
with $1 \le p \le \infty$.  If $ c < 0$, then $\exp(-tH)$
acts as a uniformly bounded semigroup on  $L^p({\mathbb R^n})$ for
$ p \in ((p_c^{\ast})', p_c^{\ast})$ and the range $((p_c^{\ast})', p_c^{\ast})$ is optimal, see for example  \cite{LSV}.
 It was proved in \cite[Section 10]{COSY} that
 $H$ satisfies restriction estimate
\begin{equation}\label{hardy11}
\big\|dE_{H}(\lambda)\big\|_{p\to p'}\leq C \lambda^{{n\over 2}({1\over p}-{1\over p'})-{1}}, \ \ \  \ \lambda> 0
\end{equation}
 for all $p \in ((p_c^{\ast})',  \frac{2n}{n+2}]$. If $c \ge 0$, then \eqref{hardy11} for
 $p = (p_c^{\ast})' = 1$ is included.

Assume that   $n=3$. Next we will use the standard perturbation techniques to prove the following result.

\begin{proposition} Suppose that $H=-\Delta+V$ on ${\mathbb R}^3$ and $V(x)={c/|x|^2}$.
Then there exists a constant $c_0>0$ such that if $0\leq c\leq c_0$, then estimate
\begin{equation}\label{hardy111}
\big\|dE_{H}(\lambda)\big\|_{p\to p'}\leq C \lambda^{{3\over 2}({1\over p}-{1\over p'})-{1}}, \ \ \  \ \lambda\ge 0
\end{equation}
 holds for all $1\leq p\leq 4/3$.
\end{proposition}
\begin{proof} Because  estimate (\ref{hardy111}) has known for all  $p \in [1, 6/5]$ by (\ref{hardy11})
when $n=3$, so it suffices to prove the spectral estimates for all $ p\in [6/5, 4/3]$.
We now start by recalling the well-known representation of the free resolvent
$R_0(z)=(-\Delta-z)^{-1}$
\begin{eqnarray*}
R_0(\zeta^2)g(x)=(-\Delta-\zeta^2)^{-1}g(x)=
\left\{
\begin{array}{ll}
{1\over 4\pi} \int_{\R^3}{e^{i\zeta|x-y|}\over |x-y|} g(y)dy \ \ \ &{\rm for}\ {\rm Im}\, \zeta >0,\\[6pt]
{1\over 4\pi} \int_{\R^3}{e^{-i\zeta|x-y|}\over |x-y|} g(y)dy \ \ \ &{\rm for}\ {\rm Im}\, \zeta <0,
\end{array}
\right.
\end{eqnarray*}
see e.g. \cite{D-H2}. By elementary computations we obtain that for $z\not=0$ and  $1<p<3/2$, it follows from
 \cite[Corollary 14]{D-H2} that
 \begin{eqnarray*}
\|V R_0(z) g\|^p_p &\leq & c\int_{\R^3}{|\Delta^{-1}(|g|)|^p\over |x|^{2p}} dx\nonumber\\
 &\leq & cK(p)\int_{\R^3}{|g|^p } dx,
\end{eqnarray*}
where $K(p)={p^{2}\over 3(3-2p)(p-1)}.$
Then there exists a constant $c_0>0$ such that when $0\leq c\leq c_0$, we
  have that
  $\|VR_0(z)\|_{p\to p}\le 1/2$ and $$\sup_{|z|>0}\|(I+VR_0(z))^{-1}\|_{p\to p}\le 2
  $$
 for  all $1<p<3/2$. Next
$$
\|R_0(z)\|_{p\to p'}\leq C|z|^{3({1\over p}-{1\over p'})-1}, \ \ \ \ z\in \mathbb{C}^\pm\setminus\{0\}
$$
for  all $6/5\leq p\leq 4/3$, see e.g. Stein\cite[P. 370]{St1}, so  required estimate
\eqref{hardy111} for $6/5\leq p\leq 4/3$  follows
  from the perturbation formula \eqref{eq3-01a} and  Stone's formula \eqref{eq3-01b}.
\end{proof}

\bigskip

\noindent
{\bf Acknowledgements:} We thank P. Chen  and A. Hassell for useful discussions.
A. Sikora was supported by
Australian Research Council  Discovery Grant DP 130101302. L. Yan was supported by
 NNSF of China (Grant No.  11371378).
 X. Yao was  supported by NSFC (Grant No. 11371158), the program for Changjiang Scholars and
Innovative Research Team in University (No. IRT13066).

  \end{document}